\documentclass{amsart}

\usepackage{amsmath,amsthm,amscd,amssymb}
\usepackage{mathrsfs}
\usepackage{graphicx}
\usepackage{enumerate}

\usepackage{ascmac}
\usepackage{color}
\usepackage{here}

\theoremstyle{plain}
\newtheorem{thm}{Theorem}[section]

\newtheorem{lem}[thm]{Lemma}
\newtheorem{prop}[thm]{Proposition}

\theoremstyle{definition}

\theoremstyle{remark}
\newtheorem{remark}[thm]{Remark}

\newtheorem*{rem}{Remark}
\numberwithin{equation}{section}

\newcounter{tmp}

\title[Abundance of observable Lyapunov irregular sets]{Abundance of observable Lyapunov irregular sets}

\date{\today}

\author{Shin Kiriki}
\address[Shin Kiriki]{Department of Mathematics, Tokai University, 4-1-1 Kitakaname, Hiratuka, Kanagawa, 259-1292, JAPAN}
\email{kiriki@tokai-u.jp}

\author{Xiaolong Li}
\address[Xiaolong Li]{Graduate School of Mathematics and Statistics, Huazhong University of Science and Technology, Luoyu Road 1037, Wuhan, CHINA}
\email{lixl@hust.edu.cn}

\author{Yushi Nakano}
\address[Yushi Nakano]{Department of Mathematics, Tokai University, 4-1-1 Kitakaname, Hiratuka, Kanagawa, 259-1292, JAPAN}
\email{yushi.nakano@tsc.u-tokai.ac.jp}

\author{Teruhiko Soma}
\address[Teruhiko Soma]{Department of Mathematical Sciences, Tokyo Metropolitan University, 1-1 Minami-Ohsawa, Hachioji, Tokyo, 192-0397, JAPAN}
\email{tsoma@tmu.ac.jp}



\begin{document}

\begin{abstract}
Lyapunov exponent is  widely used in natural science  to find   chaotic signal, 
but its existence  is seldom discussed. 
In the present paper, 
we consider the problem of whether 
 the set of points at which Lyapunov exponent fails to exist, called the  Lyapunov irregular set, has positive Lebesgue measure. 
The only known example with the Lyapunov irregular set of positive Lebesgue measure is  a figure-8 attractor
by the work of Ott and Yorke \cite{OY2008}, whose key mechanism (homoclinic loop) is easy to be broken by small perturbations.
In this paper, we show that surface diffeomorphisms with a robust homoclinic tangency given by Colli and Vargas   \cite{CV2001}, as well as other  several known  nonhyperbolic dynamics, 
 have the Lyapunov irregular set  of positive Lebesgue measure. 
We can construct such positive Lebesgue measure sets both as the time averages exist 
and   do not exist on it. 
\end{abstract}

\maketitle

\section{Introduction}\label{section:introduction}
 Lyapunov exponent is  a 
  quantity to measure   sensitivity of  an orbit  to initial conditions  
   and natural scientists often compute  it  to find  chaotic signal. 
However, the existence of Lyapunov exponent is seldom discussed. 
The aim of this paper is to investigate the abundance of dynamical systems whose Lyapunov exponents  fail to exist 
 on a physically observable set, that is, a \emph{positive Lebesgue measure} set.

Let $M$ be a compact Riemannian manifold and $f: M\to M$ a differential map.
A point $x\in M$
  is said to be \emph{Lyapunov irregular} if  there is a 
  non-zero
   vector $v\in T_xM$ such that the Lyapunov exponent of $x$ for $v$, 
\begin{equation}\label{eq:0730a}
\lim _{n\to \infty} \frac{1}{n} \log \Vert Df^n(x) v\Vert ,
\end{equation}
does not exist. 
When we would like to emphasize the dependence on $v$, we call it Lyapunov irregular for $v$.
Similarly  a point $x$ is said to be  \emph{Birkhoff irregular}   if  there is a 
continuous function $\varphi : M\to \mathbb R$ such that the time average $\lim _{n\to \infty} ( \sum _{j=0}^{n-1} \varphi (f^j (x)) )/n$ does not exist.
Otherwise, we say that $x$ is \emph{Birkhoff regular}. 
Moreover, we call the set of Lyapunov (resp.~Birkhoff) irregular points the   \emph{Lyapunov} (resp.~\emph{Birkhoff}) \emph{irregular set} of $f$.  
We borrowed these terminologies from  Abdenur-Bonatti-Crovisier \cite{ABC2011}, while they   studied the \emph{residuality} of Lyapunov/Birkhoff irregular sets, which is not the scope of the present paper.
Indeed,  the  residuality of   irregular sets is a generic property (\cite[Theorem 3.15]{ABC2011}) while the positivity of Lebesgue measure of  irregular sets does not hold for Axiom A diffeomorphisms, see e.g.~\cite{Young2002}. 
The terminology \emph{historic behavior} by Ruelle \cite{Ruelle2001} is also commonly used for the forward orbit of a point to mean that the point is Birkhoff irregular, 
in particular in the study of  the positivity of Lebesgue measure of Birkhoff irregular sets after Takens \cite{Takens2008}, 
 see e.g.~\cite{KS2017,KNS2019} and references therein.

Due to the  Oseledets multiplicative  ergodic theorem,  the Lyapunov irregular set of $f$
 is a  zero measure set 
 for any \emph{invariant} measure. 
However,  this tells nothing about whether the Lyapunov irregular set is of positive Lebesgue measure in general.
In fact,  the Birkhoff ergodic theorem  ensures that 
the Birkhoff irregular set   has zero measure 
with respect to any invariant   measure, but for a wide variety of  dynamical systems  the Birkhoff irregular set is known to have  positive Lebesgue measure,  
 see e.g.~\cite{Ruelle2001, Takens2008, KS2017, KNS2019} and references therein. 
Furthermore, the positivity of Lebesgue measure of the Birkhoff irregular set 
 for these examples are strongly related with \emph{non-hyperbolicity} of the systems, and 
   the two  complementary conjectures given by Palis \cite{Palis2000} 
   and Takens \cite{Takens2008} for the abundance of dynamics with the Birkhoff irregular set of   positive  Lebesgue measure 
     opened a deep research field in smooth dynamical systems theory.   
So, it is naturally  expected that 
 finding a large class of  dynamical systems with the Lyapunov irregular set 
 of positive Lebesgue measure
  would be a significant subject. 

Yet, the  known   example whose  Lyapunov irregular set has  positive  Lebesgue measure   is only a surface flow with an attracting homoclinic loop, 
called a figure-8 attractor (\cite{OY2008}),
see Section \ref{s:12} for details. 
However, 
the homoclinic loop is easy to be broken by small perturbations. 
Therefore, in this paper we give surface diffeomorphisms with a $\mathcal C^r$-\emph{robust homoclinic tangency} ($r\geq 2$) and the Lyapunov irregular set of positive Lebesgue measure. 
Recall that Newhouse  \cite{Newhouse79}  showed  that, when $M$ is a closed surface, any homoclinic tangency     yields a $\mathcal C^r$-diffeomorphism $f$ with a robust homoclinic tangency associated with a thick basic set $\Lambda$, 
that is, there is a  neighborhood $\mathcal O$ of $f$ in the set $\mathrm{Diff}^r(M)$ of $\mathcal C^r$-diffeomorphisms  such that for every $g\in \mathcal O$ the continuation $\Lambda _g$ of $\Lambda$  has a homoclinic tangency.
Such an open set $\mathcal O$  is called a \emph{Newhouse open set}. 

We finally remark that if $f$ is a $\mathcal C^1$-diffeomorphism whose Lyapunov irregular set has positive Lebesgue measure and $\tilde f$ is conjugate to $f$ by a $\mathcal C^1$-diffeomorphism $h$, that is, $\tilde f= h^{-1} \circ f \circ h$, then the Lyapunov irregular set of $\tilde f$ also has positive Lebesgue measure. 
Our main theorem is the following. 
\begingroup
\setcounter{tmp}{\value{thm}}
\setcounter{thm}{0}
\renewcommand\thethm{\Alph{thm}}
\begin{thm}\label{thm:main}
There exists a diffeomorphism $g$ in a Newhouse open set
 of $\mathrm{Diff}^r(M)$ of  a closed surface $M$  and $2\leq r<\infty$ such that for any small $\mathcal C^r$-neighborhood $\mathcal O$ of $g$ one can find an uncountable set $\mathcal L\subset \mathcal O$ satisfying the following: 
 \begin{itemize}
 \item[(1)] Every  $f$ and $\tilde f$ in $\mathcal L$ are not topologically conjugate if $f\neq \tilde f$;
 \item[(2)] For any $f\in \mathcal L$, there exist open sets $U_f\subset M$ and  $V_f\subset \mathbb R^2$, under the identification of $TU_f$ with $U_f \times \mathbb R^2$, such that any point $x\in U_f$ is  Lyapunov irregular for any non-zero vector $v\in V_f$. 
 \end{itemize}
 Furthermore, $\mathcal L$ can be decomposed into two uncountable sets $\mathcal R$ and $\mathcal{I}$ such that any point in $U_f$ is Birkhoff regular for each $f\in \mathcal R$ and 
 any point in $U_f$ is
  Birkhoff irregular for each $f\in \mathcal I$. 
   \end{thm} 
  \endgroup
  \setcounter{thm}{0}

      \begin{rem}[Generalization of Theorem \ref{thm:main}]
      It is a famous folklore result known to Bowen  that a surface flow with heteroclinically connected two dissipative saddle points  has the Birkhoff irregular set of  positive Lebesgue measure (see Subsection \ref{s:12}), and its precise proof was   given by Gaunersdorfer \cite{Gaunersdorfer1992}, see also Takens \cite{Takens1994}.
However, again, the heteroclinic connections are easily broken by small perturbations, and thus Takens asked in \cite{Takens2008} whether 
 the Birkhoff  irregular set can have  positive Lebesgue measure in a persistent manner. 
 In  \cite{KS2017}, the first and fourth authors affirmatively answered it by showing that there is a \emph{dense} subset of any $\mathcal C^r$-Newhouse open set  of surface  diffeomorphisms with $2\leq r<\infty$
 such that any element of the dense set has an open subset 
 in the Birkhoff irregular set, 
by 
extending the  technology developed for a special surface diffeomorphism with a robust homoclinic tangency given by  Colli and Vargas \cite{CV2001}. 
Furthermore, we adopt the Colli-Vargas diffeomorphism to prove Theorem \ref{thm:main}. 
Therefore, it is  likely that Theorem \ref{thm:main} can be extended to  surface diffeomorphisms in a dense subset of any Newhouse open set. 
The main technical difficulty might be the control of higher order terms of the return map of diffeomorphisms in the dense set, 
which do not appear for the return map  of the Colli-Vargas diffeomorphism, see the expression \eqref{eq:0812bb}.

Furthermore, the above   result \cite{KS2017} was recently   extended
in \cite{BBi2020} to the $\mathcal C^\infty$ and $\mathcal C^\omega$ categories by introducing a geometric model, and Colli-Vargas' result was   extended
 in \cite{KNS2021} to a $3$-dimensional diffeomorphism 
   with a  $\mathcal C^1$-robust homoclinic tangency derived from a blender-horseshoe.
 Hence, we expect that Theorem \ref{thm:main} holds for $r=\infty $, $\omega$ and for $r=1$ when the dimension of $M$ is three.
We also remark that \cite{LR2016, Barrientos2021}  extended  the result of \cite{KS2017} to $3$-dimensional flows and  
 higher dimensional diffeomorphisms. 
  \end{rem}

  \begin{rem}[Irregular vectors]
  Ott and Yorke \cite{OY2008} asserted that they constructed an open set $U$  any point of which is Lyapunov irregular for \emph{any} non-zero vectors, but we believe that their proof has a gap. 
  What one can immediately conclude from their argument  is   that  any point in $U$ is Lyapunov irregular  for non-zero vectors in the \emph{flow} direction (and thus, the set of irregular vectors are not observable); 
  see  Section \ref{s:12} for  details.
  In Section \ref{s:03}, we further show that a surface diffeomorphism with a figure-8 attractor introduced by Guarino-Guih\'eneuf-Santiago \cite{GGS2019} has an open set every element of which is Lyapunov irregular for \emph{any} non-zero vectors. 
  \end{rem}
  \begin{rem}[Relation with Birkhoff irregular sets]
 One can find  differences between    Birkhoff irregular sets 
  and Lyapunov irregular sets, other than Theorem \ref{thm:main},  
   in the literature.
Indeed, it was already pointed out in Ott-Yorke \cite{OY2008} that  the figure-8 attractor has a positive Lebesgue measure set on which the time averages exist but the Lyapunov exponents do not exist (see also \cite{Furman1997}). 
Conversely,  diffeomorphisms whose Birkhoff irregular set has   positive Lebesgue measure  but  Lyapunov irregular set has zero   Lebesgue measure were exhibited in \cite{CYZ2020}. 
We also remark that, in contrast to the deterministic case,  under physical noise both Birkhoff  and Lyapunov irregular sets of any diffeomorphism have zero Lebesgue measure by \cite{Araujo2000} and \cite{NNT2021}.
  \end{rem}

In the rest of Section 1, we explain that several nonhyperbolic systems in the literature also have 
 Lyapunov irregular sets of positive Lebesgue measure (see, in particular, Section \ref{s:o1}). 
The purpose of the attention to these examples  are  not   to increase the collection of 
dynamics with observable Lyapunov irregular sets, 
 but rather to   understand  the mechanism making observable Lyapunov irregular sets, which is especially discussed in Section \ref{s:o2}. 

\subsection{Other examples}\label{s:o1}
\subsubsection{Figure-8 attractor}\label{s:12}
Ott and Yorke showed in \cite{OY2008} that a figure-8 attractor has the Lyapunov irregular set of positive Lebesgue measure as follows.
Let $(f^t)_{t\in \mathbb R}$ be a smooth flow on $\mathbb R^2$ generated by a vector field $V: \mathbb R^2 \to \mathbb R^2$ with an equilibrium point $p$ of saddle type with homoclinic orbits, that is,
the unstable manifold of $p$  coincides with the stable manifold of $p$ and consists of $\{ p\}$ and  two orbits $\gamma  _1$, $\gamma _2$. 
We also assume that the loops $\gamma _1\cup \{p\}$ and $\gamma _2 \cup \{p\}$ are attracting in the sense that 
$
\alpha _- > \alpha _+, 
$
where $\alpha _+$ and $-\alpha _- $ are eigenvalues of the linearized vector field of $V$ at $p$ with $\alpha _\pm >0$. 
Due to the assumption, one can find   open sets $U_1$ and $U_2$ inside and near  the loops $\gamma _1 \cup \{p\}$ and $\gamma _2 \cup \{p\}$, respectively,  such that the $\omega$-limit set of  $(f^t(x))_{t\in \mathbb R}$ is $\gamma _i\cup \{p\} $ for all $x\in U_i$ with $i= 1, 2$.
In this setting, $\gamma _1 \cup \gamma _2 \cup \{p\}$ is called a \emph{figure-8 attractor}.

It is easy to see that
  the Birkhoff irregular set of the figure-8 attractor is empty inside $U_1\cup U_2$:
in fact, if $x\in U_1 \cup U_2$, then
\[
\lim _{t\to\infty} \frac{1}{t} \int ^t _{0} \varphi \circ f^s(x) ds =\varphi (p)\quad \text{for any  continuous function $\varphi : \mathbb R^2\to \mathbb R$}
\]
 (cf.~\cite{GGS2019}).
On the other hand, Ott and Yorke showed in \cite{OY2008} that any point $x$ in $U _1 \cup U_2$ is Lyapunov irregular for the vector $V(x)$, that is, the Lyapunov irregular set has positive Lebesgue measure (in fact, they implicitly put an additional assumption for simple calculations, 
see Section \ref{a:pt}). 

As previously mentioned, they also asserted that $x\in U _1 \cup U_2$ is  Lyapunov irregular for any non-zero vector $v$, because 
$(   \frac{1}{t} \log \det (Df^t(x) V(x) \, Df^t(x)v) )_{t\in \mathbb R}$ converges to $\alpha _+ - \alpha _-$ as $t\to \infty$. 
However, the oscillation of $( \frac{1}{t} \log \Vert Df^t(x)v\Vert )_{t\in \mathbb R}$ is not a direct consequence of this fact and  the oscillation of $( \frac{1}{t}\log \Vert Df^t(x) V(x)\Vert )_{t\in \mathbb R}$ when $v$ is not parallel to $V(x)$ because the angle between $Df^t(x) V(x)$ and  $ Df^t(x)v$ can also oscillate. 

\subsubsection{Bowen flow}\label{s:12o}
In \cite{OY2008}, Ott and Yorke also indicated  the oscillation of Lyapunov exponents for a vector along the  flow direction  for a special Bowen flow  by  a numerical experiment.
By following the argument of \cite{OY2008} for a figure-8 attractor, we can rigorously prove that the Lyapunov irregular set has positive Lebesgue measure for any Bowen flow. 

Let
$(f^t)_{t\in \mathbb R}$ be a smooth flow on $\mathbb R^2$ generated by a vector field $V: \mathbb R^2 \to \mathbb R^2$ of class $\mathcal C^{1+\alpha }$ ($\alpha >0$)  with two equilibrium points $p$ and $\hat p$  and two heteroclinic orbits $\gamma _1$ and $\gamma _2$ connecting the points, which are included
in the unstable and stable manifolds of $p$ respectively, such that the closed curve
$\gamma := \gamma _1 \cup \gamma _2 \cup \{p\} \cup \{ \hat p\}$ is attracting in the following sense: 
if we denote the expanding and
contracting eigenvalues of the linearized vector field around $p$ by $\alpha _+$ and $-\alpha _-$, and
the ones around $p_2$ by $\beta _+$ and $-\beta _-$, then 
\begin{equation*}
\alpha _-  \beta _- > \alpha _+  \beta _+ .
\end{equation*}
In this setting, one can find  an open set $U$ inside and near  the closed curve $\gamma$   such that the $\omega$-limit set of  $(f^t(x))_{t\in \mathbb R}$ is $\gamma $ for all $x\in U$. 
As explained, it was proven in \cite{Gaunersdorfer1992,Takens1994}  that any point in $U$  is Birkhoff irregular.
In fact, if $x\in U$, then one can find time sequences $(\tau _n)_{n\in \mathbb N}$, $(\hat \tau _n)_{n\in \mathbb N}$  (given in Section \ref{a:pt}) such that 
\begin{equation}\label{eq:0805b}
\begin{split}
&\lim _{n\to\infty} \frac{1}{\tau _n} \int ^{\tau _n} _{0} \varphi \circ f^s(x) ds =\frac{r \varphi (p) + \varphi (\hat p)}{1+ r}, \\
 &\lim _{n\to\infty} \frac{1}{\hat \tau _n} \int ^{\hat \tau _n} _{0} \varphi \circ f^s(x) ds =\frac{ \varphi (p) + \hat r \varphi (\hat p)}{1+ \hat r}
 \end{split}
\end{equation}
for any  continuous function $\varphi : \mathbb R^2\to \mathbb R$, where $r =\frac{\alpha _-}{\beta _+}$ and $\hat r =\frac{\beta _-}{\alpha _+}$.
According to Takens   \cite{Takens1994} we call such a flow a \emph{Bowen flow}.
We can show the following proposition for 
 the Lyapunov irregular set, whose proof will be given in Section \ref{a:pt}.
\begin{prop}\label{prop:0812c}
For the Bowen flow $(f^t)_{t\in \mathbb R}$  with the open set $U$ given above, 
any point $x$ in $U$ is Lyapunov irregular for the vector $V(x)$.
\end{prop}
\begin{rem}
For the time sequences in \eqref{eq:0805b} for which the time averages oscillate, we will see that
\begin{equation}\label{eq:0812e}
\lim _{n\to\infty} \frac{1}{\tau _n} \log \Vert D  f^{\tau _n}(x) \Vert = \lim _{n\to\infty} \frac{1}{\hat \tau _n} \log \Vert D  f^{\hat \tau _n}(x) \Vert =0
\end{equation}
for any $x\in U$. 
That is,  
 the mechanism causing the oscillation of Lyapunov exponents is different from the one
    leading to oscillation of time averages; see Section \ref{s:o2} for details.
\end{rem}

\subsubsection{
Guarino-Guih\'eneuf-Santiago's 
simple figure-8 attractor}\label{s:03}

A disadvantage of the arguments in Sections \ref{s:12} and \ref{s:12o} is that,  although it follows the arguments that a point $x$ in the open set $U_1 \cup U_2$ or $U$ is Lyapunov irregular for the vector $V(x)$ generating the flow, it is unclear whether $x$ is also Lyapunov irregular for a vector which is not parallel to $V(x)$, because the derivative $Df^t (x)$ at the return time $t$ to  neighborhoods of $p$ or $p \cup \hat p$ is not explicitly calculated  in the arguments (instead, the fact that $Df^t (x) V(x) = V(f^t(x))$ is used).  
On the other hand, Guarino, Guih\'eneuf and Santiago in \cite{GGS2019} constructed a surface diffeomorphism with a pair of saddle connections forming a figure of eight and whose return map is affine (see Proposition \eqref{eq:GGSkey}).
By virtue of this simple form of the return map, it is quite easy to prove
 that the diffeomorphism has an open set each element of which is Lyapunov irregular for \emph{any} non-zero vectors.
Furthermore, we will see in Section \ref{s:o2}  that the calculation is a prototype of the proof of Theorem \ref{thm:main}.

Fix a constant $\sigma  > 1$ and numbers $a, b$ such that $1 <  a < b < \sigma $.
Let $I=[a,b]$ and  denote the map $\mathbb R^2 \ni (x, y) \mapsto (\sigma ^{-2} x,\sigma y)$ by $H$.
For every $n\in \mathbb{N}$, 
let $S_n =I\times \sigma ^{-n} I$ and $U_n =\sigma ^{-n}I \times I$,  
so that 
\[
H^n ( S_n) = U_{2n} \quad \text{and } \quad H^n : S_n \to U_{2n} \; \text{is a diffeomorphism}.
\]
See Figure \ref{fig-GGS}.
Furthermore, let $R: \mathbb{R}^2\to \mathbb{R}^2$ be the affine map which is a rotation of  $-\frac{\pi }{2}$ around the point $( \frac{a+b}{2}, \frac{a+b}{2} )$, 
 i.e.
\[
R(x,y) = (a+b -y , x).
\]
We say that a diffeomorphism
of the plane is said to be 
\emph{compactly supported} if it equals the identity outside a ball
centered at the origin $O$, 
and moreover the diffeomorphism has a \emph{saddle (homoclinic) connection}
if it has a separatrix of the stable manifold coinciding with a separatrix of the unstable manifold associated with a saddle periodic point $O$, so that it bounds an open 2-disk. Specially, 
we call the union of $O$ and a pair of saddle connections associated with $O$ 
a \emph{figure-8 attractor} at $O$, and it satisfies $W^u(O)=W^s(O)$.

\begin{prop}[{\cite[Proposition 3.4]{GGS2019}}]\label{prop:GGS}
There exists a compactly supported $\mathcal C^\infty$-diffeomorphism $f:\mathbb R^2 \to \mathbb R^2$
which has a saddle connection of a saddle fixed point $O=(0,0)$, 
and moreover there are positive integers $n_0, k_0$ such that the following holds:
\begin{enumerate}[{\rm (a)}]
\item There is a neighborhood $\mathcal V$ of $O$  such that 
\[
\bigcup _{n\geq n_0} \bigcup _{0\leq \ell \leq n}f^\ell (S_n)\subset \mathcal V \quad \text{and } \quad f\vert _{\mathcal V} = H.
\]
\item $f^{k_0} (U_n)=S_n$ for all $n\geq n_0$ and 
\[
f ^{k_0}(x,y) =R(x,y) \quad \text{for all $(x,y) \in [0, \sigma ^{-2n_0}] \times I$}.
\]
\end{enumerate}
In particular, for every $n \geq  n_0$, 
\begin{equation}\label{eq:GGSkey}
 f^{n+k_0}(x, y) = (a + b - \sigma ^ny , \sigma ^{-2n} x) \in S_{2n} \quad \text{for all $(x, y) \in  S_n$}.
\end{equation}
\end{prop}
\begin{figure}[hbt]
\centering
\scalebox{0.8}{
\includegraphics[clip]{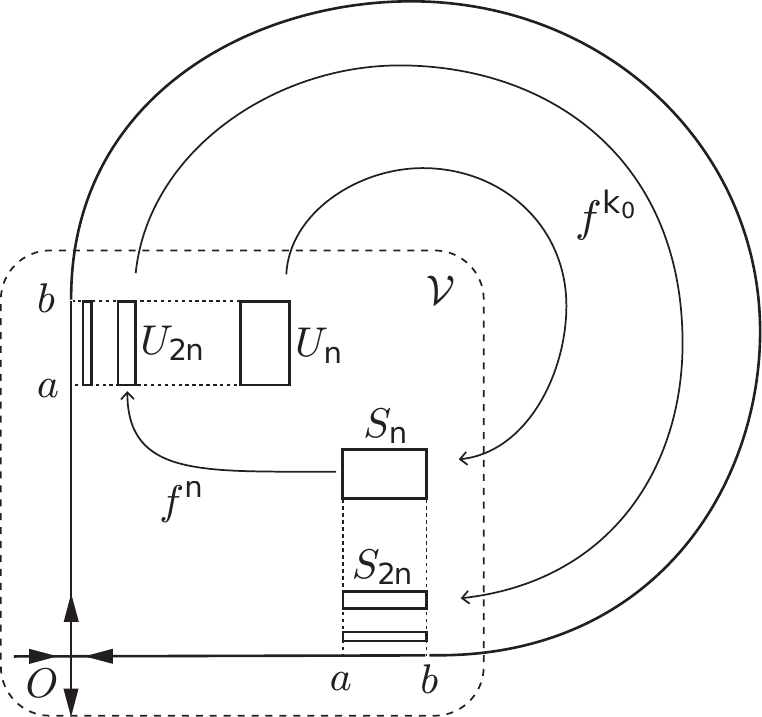}
}
\caption{ Guarino-Guih\'eneuf-Santiago's  diffeomorphism} 
\label{fig-GGS}
\end{figure}
\begin{rem}
If we suppose that $f|_{V_3} = s_h \circ f| _{V_1} \circ s_v$, 
where 
$V_i$ is the $i$-th quadrant of $\mathbb R^2$, and
$s_v, s_h : \mathbb{R}^2 \to \mathbb{R}^2$ are symmetry maps with respect to the vertical and horizontal axes, respectively, $f$ has a figure-8 attractor at $O$, see \cite{GGS2019}.
\end{rem}
Although the dynamics in  Proposition \ref{prop:GGS} is defined on $\mathbb R^2$, one can easily embed the restriction of $f$ on the support of $f$ into any compact surface.
It follows from \cite[Corollary 3.5]{GGS2019} that if $z\in S_{n_0}$, then
\[
\lim _{n\to \infty} \frac{1}{n} \sum _{j=0}^{n-1} \varphi (f^j (z)) = \varphi (O) \quad \text{for any continuous function $\varphi :\mathbb R^2\to \mathbb R$}.
\]
In particular, any point in $S_{n_0}$ is  Birkhoff regular.
Our result for the Lyapunov irregular set 
  is the following, whose proof will be given in Section \ref{s:0811b}.

\begin{thm}\label{prop:0811}
For the diffeomorphism $f$ and the rectangle $S_{n_0}$ given in Proposition \ref{prop:GGS}, any point $z$ in $S_{n_0}$ is Lyapunov irregular for any non-zero vector.
\end{thm}

\begin{rem}
We note that the piecewise expanding map on a surface constructed by Tsujii \cite{Tsujii2000} has a return map around the origin whose form is quite similar to one of the diffeomorphism of Theorem \ref{prop:0811}. 
So,  it is natural to expect that (a slightly modified version of) the   map in \cite{Tsujii2000} has an open set consisting of Lyapunov irregular points for any non-zero vectors.
\end{rem}

\subsection{Idea of proofs of Theorems \ref{thm:main} and \ref{prop:0811}: anti-diagonal matrix form of the return map}\label{s:o2}
\subsubsection{The figure-8 attractor}
We start from  the Guarino-Guih\'eneuf-Santiago's 
 figure-8 attractor. 
 Let $f$ be the diffeomorphism given in  Proposition \ref{prop:GGS}.
 Then, it follows from \eqref{eq:GGSkey}   that for any $n\geq n_0$ and $z\in S_n$, $Df^{n+k_0}(z) $ is an anti-diagonal matrix, 
\begin{equation}\label{eq:0812a}
Df^{n+k_0}(z) =\left(\begin{array}{cc}0 & -\sigma^{n} \\ \sigma^{-2n} & 0\end{array}\right),
\end{equation}
so   $Df^{(2n+k_0) +(n+k_0)}(z)  =Df^{2n+k_0}(f^{n+k_0}(z)) Df^{n+k_0}(z) $ is a diagonal matrix. 
Hence, if we define the $d$-th return time $N(d)$  from $S_{n_0}$ to  $\bigcup _{n\geq n_0} S_n$ with $d\geq 1$ by
\begin{equation}\label{eq:0812b1}
 N(d) = \sum _{d'=1}^{d} n(d'), \quad 
 n(d') = 2^{d'-1}n_0 +k_0
\end{equation}
(notice that $f^{N(d)} (S_{n_0}) \subset S_{2^d n_0}$), then it follows from a 
 chain of calculations that 
for any $z\in S_{n_0}$ 
\begin{equation}\label{eq:0812d1}
\begin{split}
& Df^{N(2d-1)}(z)=(-1)^{d-1}
\left(\begin{array}{cc}0 & -\sigma^{n_0 }\\ \sigma^{-2^{2d-1}n_{0}} & 0\end{array}\right),\\
& Df^{N(2d)}(z)=(-1)^{d}
\left(\begin{array}{cc}
1 & 0\\ 
0 & -\sigma^{(-2^{2d}+1)n_{0}}
\end{array}\right),
\end{split}
\end{equation}
and thus, for any $v\not\in \mathbb R  \left(
    \begin{array}{c}
      1\\
      0
    \end{array}
  \right) \cup \mathbb R  \left(
    \begin{array}{c}
      0\\
      1
    \end{array}
  \right) $,
\begin{equation}\label{eq:0812b1c}
\displaystyle
\lim _{d \to \infty} \frac{1}{N(d)} \log \left\Vert Df^{N(d)}(z) v \right\Vert 
=0.
\end{equation}
Furthermore,  one can see by a direct calculation that with the function $\vartheta :  [0,1]\to \mathbb R$ given by $\vartheta (\zeta )=-(1-\zeta )/(1+\zeta )$ if $\zeta \geq 1/3$ and $\vartheta (\zeta )=-2\zeta /(1+\zeta )$ if $\zeta <1/3$, it holds that for any $\zeta \in [0,1]$, 
\begin{equation}\label{eq:0812b1c2}
\lim _{d \to \infty} \frac{1}{N(4d) +\lfloor\zeta 2^{4d}n_0\rfloor} \log \left\Vert Df^{N(4d) +\lfloor\zeta 2^{4d} n_0\rfloor}(z) v \right\Vert  = \vartheta (\zeta ) \log \sigma ,
\end{equation}
where $\lfloor a \rfloor$ for $a\in \mathbb R$ is the greatest integer less than or equal to $a$. 
Note that $N(4d) =2^{4d}n_0 + (4dk_0-n_0)$, so   $N(4d) +\lfloor\zeta 2^{4d}n_0\rfloor$ over $\zeta \in [0,1]$ essentially realizes all times from $N(4d)$ to $N(4d+1)$.
A detailed calculation will be given in Section \ref{s:0811b}.

\subsubsection{The Newhouse open set}
Next we   consider the diffeomorphisms in the Newhouse open set given in Theorem \ref{thm:main}.
Colli and Vargas constructed  in \cite{CV2001} a 
 diffeomorphism 
$g$ in a Newhouse open set  with constants $0<\lambda <1 <\sigma$ such that 
for any $\mathcal C^r$-neighborhood $\mathcal O$ of $g$ and  any increasing sequence $(n_k^0)_{k\geq 0}$ of integers with  $\limsup _{k\to\infty} n_{k+1}/n_k <\infty $, 
one can find a 
 diffeomorphism $f$ in $\mathcal O$ together with  a sequence of rectangles $(R_k)_{k=1}^\infty$ 
and a sequence of increasing sequence $(\tilde n_k)_{k\geq 1}$ of integers with $\tilde n_k =O(k)$
 such that $f^{n_k+2} (R_k) \subset R_{k+1}$ and for each $(\tilde x_k+x,y)\in R_k$, 
\begin{equation}\label{eq:0812bb}
f^{n_k+2}(\tilde x_k+x,y) = (\tilde x_{k,1}-\sigma^{2n_k}x^2-\lambda^{n_k}y, \sigma^{n_k}x),
\end{equation}
where $n_k= n_k^0 + \tilde n_k$ and $(\tilde x_k,0)$ is the center of $R_k$,  see Theorem \ref{prop:0812} for details. 
Thus,  the  derivative of the return map has the form
\begin{equation}\label{eq:0812b}
Df^{n_k+2}(\tilde x_k +x, y)= \left(                 
  \begin{array}{cc}   
    -2\sigma^{2n_k}x & -\lambda^{n_k}  \\  
    \sigma^{n_k}      &  0             \\  
  \end{array}
\right).
\end{equation}
Compare this formula with \eqref{eq:0812a} for $n =n(d) -k_0$  and note that $\lim _{d\to \infty}(n(d+1) -k_0)/(n(d)-k_0) =2$.

The biggest obstacle in \eqref{eq:0812b}
 to  repeat the above calculation for Guarino-Guih\'eneuf-Santiago's  figure-8 attractor is the  term $-2\sigma ^{2n_k} x$: 
the absolute value of the term
  should be as small as the absolute value of $-\lambda ^{n_k}$ of \eqref{eq:0812b}, while $\sigma ^{2n_k}$ may be much larger than $\lambda ^{n_k}$  because $0<\lambda <1 <\sigma $.  
Therefore, the key point in the proof is to    find a subset $U_k$ of $R_k$ such that any $x\in U_k$ satisfies the required condition $\vert -2\sigma ^{2n_k} x\vert <\xi \vert -\lambda ^{n_k}\vert$ with a   positive constant $\xi $ independently of $k$  (Lemma \ref{lem4}), 
and to show $f^{n_k +2}(U_k) \subset U_{k+1}$ (Lemma \ref{lem2}).  

\subsubsection{Some technical observations}
Finally, we give a couple of (more technical) remarks on the similarity of mechanics leading to observable Lyapunov irregular sets for the dynamics of this paper. 
\begin{rem}
To understand the time scale $N(4d) +\lfloor\zeta 2^{4d}n_0\rfloor$ of \eqref{eq:0812b1c2}, calculations of (partial) Lyapunov exponents for the Bowen flow might be helpful. 
Let $(f^t)_{t\in\mathbb R}$, $V$, $p$, $\hat p$, 
$U$ be as in Section \ref{s:12o}. 
Let $N$ and $\hat N$ be  small neighborhoods of $p$ and $\hat p$, respectively,  such that $N \cap \hat N =\emptyset$.
Fix $z\in U$ and let 
$\tau _n$ and $\hat \tau _n$ be 
the  $n$-th return time of $z$ to $N$ and $\hat N$, respectively
(see Section \ref{a:pt} for their precise definition). 
Then,   since $Df^t (z)V(z) =V(f^t(z))$ for each $t\geq 0$, 
both $\Vert Df^{\tau _n} (z) V(z)\Vert $ and $\Vert Df^{\hat \tau _n} (z) V(z)\Vert $  are  bounded from above and below uniformly with respect to $n$, which implies \eqref{eq:0812e} (while \eqref{eq:0805b} is a consequence of  \cite{Takens1994}).

We further define  $\rho _n$  as the time $t$ in $[0, \tau _{n+1}-\tau _n]$ at which $f^{\tau _n+ t}(z)$ makes  the closest approach to $p$ 
  (that is, $\rho _n$ is the the minimizer of 
$
\Vert f^{\tau _n + t}(z) - p\Vert 
$
over $0\leq t \leq \tau _{n+1}-\tau _n$). 
Then, since the vector field $V$ is zero at $p$, it can be expected that $\Vert Df^{\tau _n +\rho _n} (z)V(z)\Vert  =\Vert V(f^{\tau _n +\rho _n}(z))\Vert $ decays rapidly as $n$ increases. 
In fact, we can show that 
\[
\lim _{n\to \infty} \frac{1}{\tau _n +\rho _n}\log \Vert Df^{\tau _n +\rho _n} (z)V(z)\Vert  = \frac{  \alpha _+ \beta _+ -\alpha _- \beta _-}{\alpha _+ +\beta _+ + \alpha _- + \beta _-} 
 <0,
\]
which is $ \frac{\alpha _+ - \alpha _- }{2}$ when $\alpha _+=\beta _+$ and $\alpha _- =\beta _-$.
On the other hand, $\vartheta (\zeta )$ in  \eqref{eq:0812b1c2} takes the minimum $-\frac{1}{2}$ at $\zeta =\frac{1}{3}$, so the minimum of \eqref{eq:0812b1c2} is
\[
\lim _{d \to \infty} \frac{1}{N(4d) +  \lfloor \frac{2^{4d}n_0}{3}\rfloor} \log \left\Vert Df^{N(4d) + \lfloor \frac{2^{4d}n_0}{3}\rfloor}(z) v \right\Vert  
=
-\frac{1}{2} \log \sigma 
= \frac{  \log \sigma +\log \sigma ^{-2}}{2}.
\]
\end{rem}

\begin{rem}
We emphasize that the choice of $(n_k^0)_{k\in \mathbb N}$ in \eqref{eq:0812bb} is totally  free except the condition $\limsup _{k\to\infty} n_{k+1}^0/n_k ^0 < \infty$, 
while $(n(d))_{d\in \mathbb N}$ in \eqref{eq:0812b1} must satisfy $\lim _{d\to \infty}n(d+1)/n(d) =2$. 
This freedom  makes the construction of  the oscillation of (partial) Lyapunov exponents of $f$ a bit simpler. 
Indeed, in the proof of Theorem \ref{thm:main} we take $(n_k)_{k\in \mathbb N}$ as 
\begin{equation*}
\lim _{p\to \infty} \frac{n_{2p+1}}{n_{2p}} < 
\lim _{p\to \infty}\frac{n_{2p}}{n_{2p-1} }< \infty ,
\end{equation*}
 which enables us to conclude 
   that  for any $z$ in an open subset of $R_{\kappa }$ with some large integer $\kappa$ and any vector $v$ in an open set, 
\begin{equation*}
\begin{split}
\displaystyle
\lim _{p \to \infty} \frac{1}{ N_{2p-1}  } \log \left\Vert Df^{ N_{2p-1} }(z) v \right\Vert  &= \dfrac{\log\lambda+\alpha\log\sigma}{1+\alpha}\\
<\lim _{p \to \infty} \frac{1}{N_{2p}} \log \left\Vert Df^{ N_{2p} }(z) v \right\Vert  &= \dfrac{\log\lambda+\beta\log\sigma}{1+\beta},
\end{split}
\end{equation*}
where $\alpha =
\lim _{p\to \infty}  n_{2p+1}/n_{2p}$, $\beta =
\lim _{p\to \infty} n_{2p}/n_{2p-1} $ and $N_j =(n_\kappa  +2) + (n_{\kappa +1} +2)  +\cdots + (n_{\kappa +j} +2)$
(so the ``time at closest approach'' $N(4d)+\lfloor \zeta 2^{4d}n_0\rfloor$ with $\zeta \in (0,1)$  for   Guarino-Guih\'eneuf-Santiago's  figure-8 attractor is not necessary).
\end{rem}

\begin{rem}
We outline why the open set $V_f$ in Theorem \ref{thm:main} is not easy to be replaced by  $\mathbb R^2\setminus\{0\}$ by our argument.
Again, Guarino-Guih\'eneuf-Santiago's  figure-8 attractor might be useful to understand the situation.
Let $v$ be the unit vertical vector. Then, it follows from \eqref{eq:0812d1} that 
 $\left\Vert Df^{N(2d)}(z) v \right\Vert = \sigma ^{(-2^{2d} +1)n_0}$, which is much smaller than the lower bound $1- \sigma ^{(-2^{2d} +1)n_0}$ of  $\left\Vert Df^{N(2d)}(z) v' \right\Vert $ for any non-zero vector $v'$ being not parallel to $v$, and thus   \eqref{eq:0812b1c} does not hold for this $v$ (see \eqref{eq:0813d} for details).
For the diffeomorphism of Theorem \ref{thm:main},
this special situation on the vertical line 
 may be spread to a vertical cone $\mathcal K_v:=\{ (v_1, v_2) \in \mathbb R^2\mid \vert v_1 \vert \leq K ^{-1} \vert v_2\vert \}$ with a constant $K>1$ (see \eqref{eq:0915c})  and it is hard to repeat the above calculation   on the cone due to  the higher order term $-2\sigma ^{2n_k}x$ of \eqref{eq:0812b}. 
A similar difficulty occurs on a horizontal  cone $\mathcal K_h:=\{ (v_1, v_2) \in \mathbb R^2\mid \vert v_2 \vert \leq  K ^{-1} \vert v_1\vert \}$, and the open set $V_f$  of  Theorem \ref{thm:main} is given as $\mathbb R\setminus (K_v \cup K_h)$.
\end{rem}

\section{Proof of Proposition \ref{prop:0812c}}\label{a:pt}
We follow the argument \cite{OY2008} for the 
 figure-8 attractor,\footnote{They implicitly ignored the higher order terms of 
   the transient map of the flow, i.e.~assumed that $\hat s_n =cr_n$ and $s_{n+1}=\hat c\hat r_n$ instead of  \eqref{eq:0813c} below.}
  so the reader familiar with this subject can skip this section. 
  Let $(f^t)_{t\in \mathbb R}$ be the Bowen flow given in Section \ref{s:12o}. 
Let $N$ and $\hat N$ be neighborhoods of $p$ and $\hat p$, respectively, such that 
there are linearizing coordinates $\phi : N\to \mathbb R^2$ and $\hat \phi : \hat N\to 
 \mathbb R^2$ satisfying that  both $\phi (N )$ and $\hat \phi (\hat N )$ include   $(0,1]^2$ and 
\begin{equation}\label{eq:0724b}
\phi \circ f^t\circ \phi ^{-1}(r,s)=(e^{-\alpha _-t}r, e^{\alpha _+t}s),\quad
\hat \phi \circ  f^t\circ \hat \phi ^{-1}(r,s)=(e^{-\beta _-t}r, e^{\beta _+t}s)
\end{equation}
on $  (0,1]^2$.  
Fix $(x,y)\in U$.
Let $\hat T_0$ be the hitting time of  $(x,y)$ to $\{   \phi ^{-1}(1,s) \mid s\in (0,1]\}$, i.e.~the smallest positive number $t$ such that   $   f^{t}(x,y) = \phi ^{-1}(1,s)$ with some $s\in (0,1]$. 
Let $s_1$ be the second component of $\phi \circ f^{\hat T_0}(x,y)$.
We inductively define sequences  $(t_n, T_n, \hat t_n, \hat T_n)_{n\in \mathbb N}$, $(s_n, r_n, \hat s_n, \hat r_n)_{n\in \mathbb N}$ of positive numbers as
\begin{itemize}
\item $t_n$ is the hitting time of $\phi^{-1}(1,s_n)$ to $\{ \phi ^{-1}(r,1) \mid r\in (0,1]\}$, 
and $r_n$ is the first component of $\phi \circ f^{t_n}\circ \phi ^{-1}(1,s_n) $, 
\item $T_n$  is the hitting time of  $\phi ^{-1}(r_n,1)$ to $\{ \hat \phi ^{-1}(1,  s) \mid  s\in (0,1]\}$, and $\hat s_n$ is the  second component of $\hat \phi \circ f^{T_n}\circ \phi ^{-1}(r_n,1)$, 
\item $\hat t_n$ is the hitting time of $\hat \phi^{-1}(1,\hat s_n)$ to $\{ \hat \phi ^{-1}(r,1) \mid r\in (0,1]\}$, 
and $\hat r_n$ is the first component of $\hat \phi \circ f^{t_n}\circ \hat \phi ^{-1}(1,\hat s_n) $, 
\item $\hat T_n$  is the hitting time of  $\hat \phi ^{-1}(\hat r_n,1)$ to $\{  \phi ^{-1}(1,  s) \mid  s\in (0,1]\}$, and $ s_{n+1}$ is the  second component of $  \phi \circ f^{T_n}\circ \hat \phi ^{-1}(\hat r_n,1)$. 
\end{itemize}
Then, from 
\[
(e^{-\alpha _- t_n}, e^{\alpha _+t_n}s_n) =(r_n,1), \quad (e^{-\beta _- \hat t_n}, e^{\beta  _+\hat t_n}\hat s_n) =(\hat r_n,1)
\]
it follows that
\begin{equation}\label{eq:-724}
t_n=- \frac{\log s_n}{\alpha _+}, \quad r_n = s_n^{a}, \quad \hat t_n=- \frac{\log \hat s_n}{\beta _+}, \quad \hat r_n = \hat s_n^{b }.
\end{equation}
with $a:=\frac{\alpha _- }{\alpha _+}$ and $b:=\frac{\beta _-}{\beta _+}$. 
On the other hand, it is straightforward to see that both $T_n$ and $\hat T_n$ are bounded from above and below uniformly with respect to $n$, 
and thus, since the vector field $V$ is of class $\mathcal C^{1+\alpha }$, 
one can find positive numbers  $c$ and $\hat c$   (which are independent  of $n$) such that
\begin{equation}\label{eq:0813c}
\hat s_n = c r_n +o(r_n^{1+\alpha}), \quad   s_{n+1} = \hat c \hat r_n  +o( \hat r_n^{1+\alpha }).
\end{equation}
Moreover, we set 
\[
\tau _n:= \hat T_0 + \sum _{k=1}^{n-1} (t_k + T_k + \hat t_k + \hat T_k), \quad \hat \tau _n:= \hat T_0 + \sum _{k=1}^{n-1} (t_k + T_k + \hat t_k + \hat T_k) + t_n +T_n,
\]
that is, the $n$-th return time to $N$ and $\hat N$, respectively.
Notice that $D f^t (x,y) V(x,y) = V(f^t (x,y))$ for each $t\geq 0$. 
Hence, we have
\[
\lim _{n\to \infty}  \frac{1}{\tau _n}\log \Vert D f^{\tau _n} (x,y) V(x,y) \Vert =\lim _{n\to \infty}  \frac{1}{\tau _n}\log \Vert   V(1,s_n) \Vert =0
\]
because $ \Vert   V(1,s) \Vert $ is bounded from above and below uniformly with respect to $s\in (0,1]$.

From now on, we identify $\phi (x,y)$ and $\hat \phi (x,y)$ with $(x,y)$ if it makes no confusion.
We further define a sequence $(\rho_n )_{n\in \mathbb N}$  of positive numbers as
$\rho _n$ is the minimizer of 
\[
\Vert f^{t}(1,s_n) - p\Vert ^2= e^{-2\alpha _- t}  + e^{2\alpha _+ t}s_n^2
\]
(under the linearizing  coordinate $\phi$) 
over $0\leq t \leq  t_n$, that is, the time at which $ f^{t}(1,s_n)$ makes the closest approach to $p$ over $0\leq t \leq  t_n$. 
Then, it follows from a straightforward calculation that
\begin{equation}\label{eq:0724c}
\rho _n = - \frac{\log s_n}{\alpha _+ + \alpha _-} +C_1,\quad \Vert L(f^{\rho_n}(1,s_n))\Vert =C_1'
s_n ^{\alpha _-/(\alpha _+ +\alpha _-)} ,
\end{equation}
where $C_1 := \frac{\log \alpha _- - \log \alpha _+}{2(\alpha _+ + \alpha _-)} $, $C_1':=\sqrt{\alpha _-^2e^{-2\alpha _- C_1} + \alpha _+^2 e^{2\alpha _+C_1}}$ and $L$ is the linearized vector sub-field of $V$ around $p$ corresponding to \eqref{eq:0724b}, i.e~$L(x,y) =(-\alpha _-x, \alpha _+y)$.
We  show that 
\begin{equation}\label{eq:0813a2a}
\limsup _{n\to\infty} \frac{1}{\tau _n +\rho _n}\log \Vert D f^{\tau _n +\rho _n} (x,y) V(x,y) \Vert
\leq \frac{   \alpha _+ \beta _+ -\alpha _-\beta _-}{\alpha _+ + \beta _+ + \alpha _- + \beta _-}.
\end{equation}
Fix $\epsilon >0$. Then, it follows from \eqref{eq:0813c} that  one can find $n_0$ such that 
\[
 c_- r_n \leq  \hat s_n \leq c_+  r_n, \quad \hat c_- \hat r_n \leq s_{n+1} \leq  \hat c_+ \hat r_n 
\]
for any $n\geq n_0$, where $c_\pm  =(1\pm\epsilon )c$ and $\hat c_\pm =(1\pm \epsilon )\hat c$.
Therefore, by  induction, together with \eqref{eq:-724}, it is straightforward to see that 
\begin{align*}
&c^{b\Lambda _n}_-  \hat c^{\Lambda _n}_-  s_{n_0 } ^{(ab)^n}\leq s_{n_0 +n} \leq (c^b_+)^{\Lambda _n}  \hat c^{\Lambda _n}_+  s_{n_0 } ^{(ab)^n}, \\
& 
 c_-  \left(c_-^{b\Lambda _{n}}  \hat c^{\Lambda _n}_-  s_{n_0 }^{(ab)^n}\right)^a
\leq \hat s_{n_0+n} \leq 
 c_+ \left(c_+^{b\Lambda _{n}}  \hat c^{\Lambda _n}_+  s_{n_0 }^{(ab)^n}\right)^a
\end{align*}
for any $n\geq 0$,
where $\Lambda _n =1+ab +\cdots +(ab)^{n-1} = \frac{(ab)^n -1}{ab-1}$.
Fix $n\geq n_0$ and write $N:=n_0+n$ to avoid heavy notations. 
Then,  it holds that
\begin{align*}
&
(ab)^n \log \left( s_{n_0} C_- \right) -C_2
  \leq \log s_{N} \leq (ab)^n \log \left(  s_{n_0} C_+\right) 
  +C_2,\\
& a(ab)^n \log \left(  s_{n_0} C_- \right) -C_2
\leq \log \hat s_{N} \leq a(ab)^n \log \left(  s_{n_0}   C_+ \right) 
+C_2
\end{align*}
with some constant $C_2>0$, where $C_\pm :=c_\pm  ^{b/(ab-1)}  \hat c_\pm ^{1/(ab-1)} $.
Thus,  by \eqref{eq:-724} we have
\begin{align*}
\tau_{N} &\geq   \sum _{k=1}^{n-1} \left(-\frac{1}{\alpha _+} - \frac{a}{\beta _+}\right) (ab)^k \log \left( s_{n_0} C_+\right)  +C_{n_0} + n C_3  \\
&=- \frac{\alpha _- + \beta _+}{\alpha _- \beta _- -\alpha _+\beta _+}  (ab)^{n}   \log \left( s_{n_0} C_+\right) +C_{n_0}' + n C_3 \end{align*}
with some constants $C_{n_0}$, $C_{n_0}'$ and $C_3$.
Furthermore, it follows from \eqref{eq:0724c} that 
\[
\rho _{N} \geq - \frac{(ab)^n \log \left(  s_{n_0}   C_+ \right) 
}{\alpha _+ + \alpha _-} +C_3',
\]
so that
\[
\tau_{N}+\rho _N\geq C_{n_0}^{\prime \prime} +nC_3 + \frac{\alpha _-(\alpha _+ + \beta _+ + \alpha _- + \beta _-)}{( \alpha _+ \beta _+ -\alpha _-\beta _-)(\alpha _+ + \alpha _-)}(ab)^n \log \left(  s_{n_0}   C_+ \right) 
\]
with some constants $C_3'$, $C_{n_0}^{\prime \prime}$.
On the other hand, by \eqref{eq:0724c} it holds that
\[
\log \Vert V(f^{\tau _{N} + \rho_{N}}(x,y))\Vert  = \log \Vert L(f^{ \rho_{N}(1,s_{N})})\Vert \leq 
\frac{\alpha _-}{(\alpha _+ +\alpha _-) }  (ab)^n \log \left(  s_{n_0} C_-\right)  + C_3'
\]
with some constants $C_3$, $C_3'$.
Therefore,
\[
\limsup _{n\to\infty} \frac{1}{\tau _n +\rho _n}\log \Vert D f^{\tau _n +\rho _n} (x,y) V(x,y) \Vert\leq \frac{ \alpha _+ \beta _+ -\alpha _-\beta _- }{\alpha _+ + \beta _+ + \alpha _- + \beta _-} \cdot \frac{\log (s_{n_0}C_-)}{\log (s_{n_0}C_+)}.
\]
Since $\epsilon$ is arbitrary, we get \eqref{eq:0813a2a} (notice that $\frac{\log (s_{n_0}C_-)}{\log (s_{n_0}C_+)}$ converges to $1$ from below as $\epsilon$ goes to zero).
In a similar manner, one can show that 
\[
\liminf _{n\to\infty} \frac{1}{\tau _n +\rho _n}\log \Vert D f^{\tau _n +\rho _n} (x,y) V(x,y) \Vert
\geq \frac{ \alpha _+ \beta _+ -\alpha _-\beta _-}{\alpha _+ + \beta _+ + \alpha _- + \beta _-},
\]
and we complete the proof of 
 Proposition \ref{prop:0812c}.
\qed

\section{Proof of Theorem \ref{prop:0811}}\label{s:0811b}
Let $f$ be the  Guarino-Guih\'eneuf-Santiago diffeomorphism of Proposition \ref{prop:GGS} and $N(d)$  the $d$-th return time  given in \eqref{eq:0812b1}.
Fix $z\in S_{n_{0}}$. 
By induction with respect to $d$ 
we first show \eqref{eq:0812d1}. 
It   immediately follows from \eqref{eq:GGSkey} 
that the first equality  of \eqref{eq:0812d1} is true for $d=1$. 
Then let us assume that 
 the first equality  of \eqref{eq:0812d1}  is true for a given positive integer $d$. 
Since 
$N(2(d+1)-1)=N(2d+1)=N(2d-1)+n(2d)+n(2d+1)$, 
by the chain rule and the inductive hypothesis, 
\begin{multline*}
Df^{N(2(d+1)-1)}(z)=Df^{n(2d+1)}(f^{N(2d)}(z)) Df^{n(2d)}(f^{N(2d-1)}(z))Df^{N(2d-1)}(z)\nonumber\\
=
\left(\begin{array}{cc}
		0 & -\sigma^{2^{2d}n_0 }\\ 
      \sigma^{-2^{2d+1}n_{0}} & 0
      \end{array}
\right)\left(\begin{array}{cc}
		0 & -\sigma^{2^{2d-1}n_0 }\\ 
      \sigma^{-2^{2d}n_{0}} & 0
      \end{array}
\right)\\
\times (-1)^{d-1}
\left(\begin{array}{cc}
		0 & -\sigma^{n_0 }\\ 
      \sigma^{-2^{2d-1}n_{0}} & 0
      \end{array}
\right)
\nonumber\\
=(-1)^{d-1}
\left(\begin{array}{cc}
		0 & \sigma^{n_0}\\ 
      -\sigma^{-2^{2d+1}n_{0}} & 0
      \end{array}
\right)
=(-1)^{d}
\left(\begin{array}{cc}
		0 & -\sigma^{n_0}\\ 
      \sigma^{-2^{2d+1}n_{0}} & 0
      \end{array}
\right).
\nonumber
\end{multline*}
That is,  the first equality  of \eqref{eq:0812d1}  holds for $d+1$. 
In a similar manner,  by induction with respect to $d$, we can prove  the  second equality  of \eqref{eq:0812d1}. 

We next prove that $z$ is Lyapunov irregular for 
any nonzero horizontal  vector $v= \left(\begin{array}{c}
		s\\ 
0
      \end{array}
\right)$. By the first equality  of \eqref{eq:0812d1}, we obtain 
\begin{equation}\label{eq:0813d}
\frac{\log \left\Vert Df^{N(2d-1)}(z) v \right\Vert}{N(2d-1)} 
  =
  \frac{ -2^{2d-1}n_{0} \log\sigma+\log |s|}{(2^{2d-1}-1)n_{0}+(2d-1)k_{0}} 
 \xrightarrow[d\to \infty]{} -\log \sigma.
\end{equation}
On the other hand, it follows from the second equality  of \eqref{eq:0812d1}  that  
\[
\frac{\log \left\Vert Df^{N(2d)}(z) v \right\Vert}{N(2d)} 
=\frac{\log |s|}{{(2^{2d}-1)n_{0}+(2d)k_{0}}}
 \xrightarrow[d\to \infty]{} 0.
\]
In a similar manner, we can show that $z$ is Lyapunov irregular for 
any nonzero vertical vector. 

Finally, we will prove \eqref{eq:0812b1c} and  \eqref{eq:0812b1c2}, which immediately implies that   $z$ is Lyapunov irregular for 
any nonzero vector $v\not \in  \mathbb R \left(\begin{array}{c}
		1\\ 
0
      \end{array}
\right) \cup  \mathbb R \left(\begin{array}{c}
		0\\ 
1
      \end{array}
\right)$. 
For simplicity, we assume that $\zeta 2^{4d} n_0$ is an integer. 
Essentially, the proof of  \eqref{eq:0812b1c} is included in the discussion until now.
Thus, we show  \eqref{eq:0812b1c2}. 
By  \eqref{eq:0812d1} and the item (a) of Proposition \ref{prop:GGS}, 
\begin{align*}
Df^{N(4d)+\zeta 2^{ 4d}n_0}(z)&=
\left(\begin{array}{cc}
\sigma^{-2\zeta \cdot 2^{4d}n_0} & 0\\ 
0 & -\sigma^{ -(1-\zeta )2^{4d} n_{0} +n_0}
\end{array}\right)\\
&=\sigma^{ -(1-\zeta )2^{4d} n_{0} }
\left(\begin{array}{cc}
\sigma^{(1-3\zeta ) 2^{4d}n_0} & 0\\ 
0 & -\sigma^{ n_0}
\end{array}\right).
\end{align*}
Fix a  vector $v= \left(\begin{array}{c}
		s\\ 
u
      \end{array}
\right) $ with $su\neq 0$.
If $1-3\zeta \leq 0$, then 
\[
\lim _{d\to \infty}\left\Vert \left(\begin{array}{cc}
\sigma^{(1-3\zeta )  2^{4d}n_0} & 0\\ 
0 & -\sigma^{ n_0}
\end{array}\right) v\right\Vert = 1.
\]
Hence, since $N(4d) +\zeta 2^{4d}n_0 = (1+\zeta )2^{4d}n_0 +(4dk_0 -n_0)$, we get
\[
\lim _{d\to \infty}\frac{1}{N(4d) +\zeta 2^{4d}n_0} \log \left\Vert Df^{N(4d) +\zeta 2^{4d} n_0}(z) v \right\Vert = -\frac{1-\zeta }{1+\zeta } \log \sigma .
\]
On the other hand, if $1-3\zeta > 0$, then 
\[
\lim _{d\to \infty}\left\Vert \left(\begin{array}{cc}
\sigma^{(1-3\zeta )  2^{4d}n_0} & 0\\ 
0 & -\sigma^{ n_0}
\end{array}\right) v\right\Vert \cdot \sigma^{-(1-3\zeta )  2^{4d}n_0} =1.
\]
Thus we get
\begin{align*}
\lim _{d\to \infty}\frac{1}{N(4d) +\zeta 2^{4d}n_0} \log \left\Vert Df^{N(4d) +\zeta 2^{4d} n_0}(z) v \right\Vert &= \frac{-(1-\zeta ) +(1-3\zeta )}{1+\zeta } \log \sigma \\
&=- \frac{2\zeta }{1+\zeta } \log \sigma .
\end{align*}
This completes the proof of Theorem \ref{prop:0811}.
\qed

\section{Proof of Theorem \ref{thm:main}}

In this section, we give the proof of Theorem \ref{thm:main}.
In Section \ref{s:4.1} we briefly recall a small perturbation    of a diffeomorphism with 
 a robust homoclinic tangency introduced by Colli and Vargas \cite{CV2001}.
In Section  \ref{s:4.2} we establish key lemmas to control the higher order term in \eqref{eq:0812a}, and prove the positivity of Lebesgue measure of Lyapunov irregular sets in Section \ref{s:4.3}.
Finally, in Section \ref{s:4.4}, we discuss the Birkhoff (ir)regularity of the set.
\subsection{Dynamics}\label{s:4.1}
Let us start the proof of Theorem \ref{thm:main} by remembering  the Colli-Vargas model with a robust homoclinic tangency introduced  in \cite{CV2001}.
The reader familiar with this subject can skip this section.  
Let $M$ be a closed surface including $[-2,2]^2$, and a diffeomorphism $g\equiv g_\mu: M\to M$ with a real number $\mu$ satisfying the following.
\begin{itemize}
\item (Affine horseshoe) There exist constants $0<\lambda<\frac{1}{2} $ and $\sigma>2$ such that   
\[
g(x, y)=\left(  
\pm \sigma \left(x\pm \frac{1}{2}\right) , \pm \lambda y  \mp \frac{1}{2}\right)\quad \text{if $\displaystyle \left\vert x\pm \frac{1}{2} \right\vert \leq \frac{1}{\sigma}$, $\vert y\vert \leq 1$}
\]
and $\lambda\sigma^2<1$; 
\item (Quadratic tangency) 
For any $(x,y)$ near a small neighborhood of $(0,-1)$,
$$
g ^{2}(x,y)=(\mu  -x^2 -y  ,x).
$$
\end{itemize}
Then, it was proven by Newhouse \cite{Newhouse1970} that there is a $\mu$ such that $g$ has a $\mathcal C^2$-robust homoclinic tangency on $\{y=0\}$. 
See Figure \ref{fig1-1}.

\begin{figure}[hbt]
\centering
\scalebox{0.7}{\includegraphics[clip]{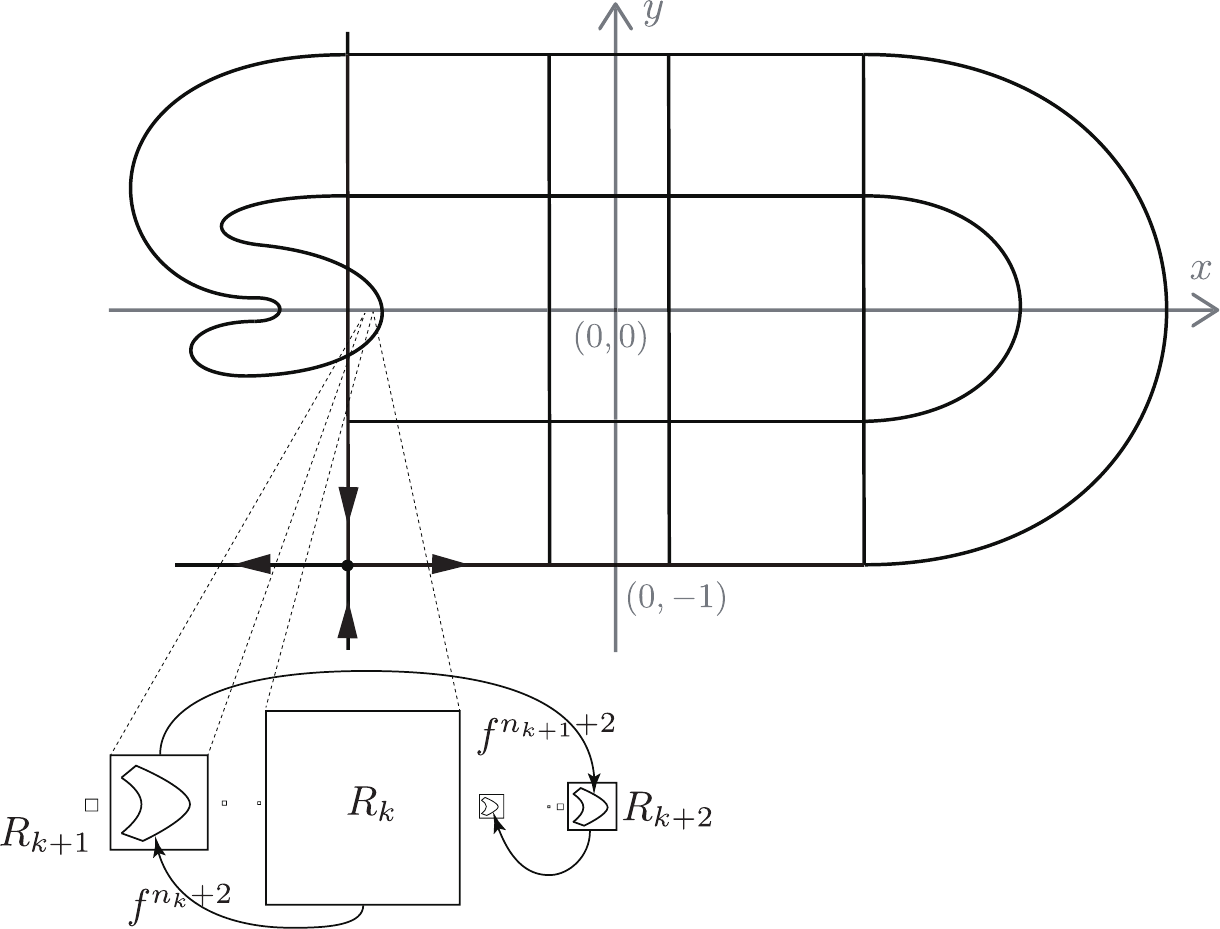}}
\caption{Colli-Vargas' diffeomorphism}
\label{fig1-1}
\end{figure}

Colli and Vargas showed the following.
\begin{thm}[\cite{CV2001}]\label{prop:0812}
Let $g $ be the surface diffeomorphism with a robust homoclinic tangency given above.
Then, for any $\mathcal C^r$-neighborhood $\mathcal O$ of $g$ $(2\leq r<\infty)$ and any increasing sequence $(n_k^0)_{k\in \mathbb N}$ of integers satisfying $n_{k}^0 =O((1+\eta )^k) $ with some  $\eta >0$, 
one can find a   diffeomorphism $f$ in $\mathcal O$ together with a sequence of rectangles $(R_k)_{k\in \mathbb N}$ and  an increasing sequence $(\tilde  n_k)_{k\in \mathbb N}$  of integers, satisfying that $\tilde n_{k} =O(k) $ and depends only on $\mathcal O$, such that the following holds for each $k\in \mathbb N$ with $n_k:= n_k^0 + \tilde n_k$:
\begin{itemize}
\item[$\mathrm{(a)}$]  $f^{n_k+2} (R_k)\subset R_{k+1}$;
\item[$\mathrm{(b)}$] For each $(\tilde x_k+x,y)\in R_k$,  
\begin{equation*}
f^{n_k+2} (\tilde x_k+x,y) = (\tilde x_{k+1}-\sigma^{2n_k}x^2\mp \lambda^{n_k}y, \pm \sigma^{n_k}x),
\end{equation*}
where $(\tilde x_k,0)$ is the center of $R_k$. 
\end{itemize} 
\end{thm}

Refer to the ``Conclusion'' given in p.~1674  and the ``Rectangle lemma'' and its proof given in pp.~1975--1976 of  the paper \cite{CV2001}, where the notation $R_k$ was used to denote a slightly different object that we will not use, and our $R_k$  was written as $R_k^*$.
See Remark \ref{rmk:0911c} and Theorem \ref{thm:0911b} for more information.

By the coordinate translation $T_k: (x,y) \mapsto (x-\tilde x_k, y)$, 
which sends $(\tilde x_k,0)$ to $(0,0)$,
 the action of $ f^{n_k+2}\vert _{R_k}$ can be rewritten as 
\begin{equation}\label{map0}
F_k: \left(                 
  \begin{array}{c}   
    x \\  
    y \\  
  \end{array}
\right) \mapsto  \left(                 
  \begin{array}{c}   
    -\sigma^{2n_k}x^2 \mp \lambda^{n_k}y  \\  
 \pm    \sigma^{n_k}x  \\  
  \end{array}
\right), 
\end{equation}
which sends $(0,0)$ to $(0,0)$,
that is, 
\[
f^{n_k+2}(x,y) =T_{k+1}^{-1} \circ F_k  \circ T_k(x,y) \quad \text{for every $(x,y)\in R_k$}.
\]
Note that for each $l\geq k$, 
\[
f^{n_l+2} \circ f^{n_{l-1}+2} \circ \cdots \circ f^{n_k+2} =T_{l+1}^{-1} \circ  \left(F_l\circ F_{l-1}  \circ \cdots   \circ F_k \right) \circ T_k,
\]
so the oscillation of $( \frac{1}{n} \log \Vert Df^n (\boldsymbol{x}) \boldsymbol{v}\Vert )_{n\in \mathbb N}$ for each $\boldsymbol{x}\in R_k$ with some $k$ and each nonzero vectors $\boldsymbol{v}$ in an open set follows from
 the oscillation of 
 \[
\left(  \frac{1}{(n_k +2 ) +  \cdots + (n_{l-1} +2) + (n_l +2)} \log \Vert  D\left(F_l\circ F_{l-1}  \circ \cdots   \circ F_k \right)(\boldsymbol{x}) \boldsymbol{v}\Vert \right) _{l\in \mathbb N}
  \]
  for each $\boldsymbol{x}\in T_k(R_k)$ and  each nonzero vectors $\boldsymbol{v}$ in the open set, 
which we will   show  in the following.

\subsection{Key lemmas}\label{s:4.2}
First, let us fix some constants in advance.
Fix a small neighborhood $\mathcal O$ of $g$, and let $(\tilde n_k) _{k\in \mathbb N}$ be the sequece given in Theorem \ref{prop:0812}.
Notice that $\lambda\sigma<\lambda\sigma^2<1$.
Take a sufficiently small $\eta>0$ and a sufficiently large integer $n_0\geq 2$ so that
\begin{equation*}
\lambda\sigma^{\frac{1+3\eta+8n_0^{-1}}{1-\eta}}<1, 
\end{equation*}
and fix $1<\alpha<\beta<1+\eta$ such that 
\begin{equation}\label{Q02}
\lambda\sigma^{\frac{6\beta-4+8n^{-1}_0}{2-\beta}}<1,  \quad \alpha^2 \beta^2<2 \quad \mbox{and} \quad  \lambda\sigma^{\alpha}<1.
\end{equation}
Let  $(n_k^0)_{k\in \mathbb N}$ be an increasing sequence  of integers given by  
\begin{equation}\label{eq:0915f}
n_{2p}^0=\lfloor n_0\alpha^p\beta^p\rfloor -\tilde n_{2p},\quad   n_{2p+1}^0=\lfloor n_0\alpha^{p+1}\beta^p\rfloor  -\tilde n_{2p+1},
\end{equation}
which are natural numbers for each $p$ by increasing $n_0$ if necessary.
Since $x-1< \lfloor  x\rfloor \le x$ and $\tilde n_k =O(k)$,
 by increasing $n_0$ if necessary, we have
\begin{align*}
\dfrac{n_{2p+1}^0}{n_{2p}^0}&<\dfrac{n_0\alpha^{p+1}\beta^p - \tilde n_{2p+1}}{n_0\alpha^p\beta^p-1- \tilde n_{2p}} < \alpha+\dfrac{\alpha (1+ \tilde n_{2p} )  }{n_0\alpha^p\beta^p-1-\tilde n_{2p}}<1+\eta,\\
\dfrac{n_{2p+2}}{n_{2p+1}}&<\dfrac{n_0\alpha^{p+1}\beta^{p+1}  - \tilde n_{2p+2}}{n_0\alpha^{p+1}\beta^p-1  - \tilde n_{2p+1}}=\beta+\dfrac{\beta (1+   \tilde n_{2p+1})}{n_0\alpha^{p+1}\beta^p- 1  - \tilde n_{2p+1}}<1+\eta ,
\end{align*}
so it holds that $n_k^0 = O((1+\eta ) ^k)$, which is the only requirement  to apply Theorem \ref{prop:0812}.
Set 
$
n_k =  n_k^0 + \tilde n_k,
$ then we obviously have
\begin{equation*}
n_{2p} =\lfloor n_0\alpha^p\beta^p\rfloor ,\quad   n_{2p+1} =\lfloor n_0\alpha^{p+1}\beta^p\rfloor .
\end{equation*}

Define  sequences   $(b_k)_{k\in \mathbb N}$ and $(\varepsilon _k)_{k\in \mathbb N}$ of positive numbers by
\begin{equation*}
b_k=\sigma^{-\sum_{i=-1}^{+\infty}\frac{n_{k+1+i}}{2^i}}
\end{equation*}  
and
\begin{equation*}
\varepsilon_k =\Big(\lambda\sigma^{\frac{6\beta-4+8n^{-1}_k}{2-\beta}}\Big)^{n_k}.
\end{equation*} 
\begin{remark}\label{rmk:0911c}
Define $\tilde b_k$ by 
\[
\tilde b_k=\sigma^{-\sum_{i=0}^{+\infty}\frac{n_{k+1+i}}{2^i}},
\]
then $R_k$ of Theorem \ref{prop:0812} is of the form
\[
R_k =\left[\tilde x_k - c_k \tilde b_k  , \tilde x_k +  c_k \tilde b_k \right] \times \left[-20 \tilde b_k^{\frac{1}{2}}, 20\tilde b_k^{\frac{1}{2}}\right] 
\]
with some constant  $c_k$ satisfying that 
\[
\frac{1}{2} \leq c_k \leq 10,
\]
see the ``Rectangle lemma'' and its proof given in pp.~1975-1976 of \cite{CV2001} (as previously mentioned,  in the paper our $R_k$ is written as $R_k^*$ and the notations $R_k$ is used for another object).
Note that $b_k<\tilde b_k$. Thus,  $F_k$ in \eqref{map0} is well-defined on any rectangle of the form
\[
\left[- c  b_k  ,   c   b_k \right] \times \left[- c  \sqrt{b_k}, c \sqrt{b_k }\right] \quad \text{with $\displaystyle 0< c \leq \frac{1}{2}$}.
\]
In the paper \cite{CV2001}  the notation $b_k$ was used to denote $\tilde b_k$, but  this positive number is not explicitly used in the following argument, so  we defined $b_k$ as above for notational simplicity.
\end{remark}
By the construction of $(n_k) _{k\in \mathbb N}$, we have that $n_l / (n_k +1) < \beta ^{l -k}$ for each $k\leq l$.
Hence, since $n_k$ is increasing,
\begin{align*}
4n_k&<2n_k+n_{k+1}+ \frac{n_{k+2}}{2} + \cdots \\
&< 2(n_k +1) \left(1+ \frac{\beta}{2} + \frac{\beta ^2}{2^2} +\cdots \right) 
=
\dfrac{4(n_k+1)}{2-\beta}.
\end{align*}
Therefore we have
\begin{align}\label{Q05}
\sigma^{-\frac{4(n_k+1)}{2-\beta}} < b_k < \sigma^{-4n_k}\quad \text{for each $k\in\mathbb{N}$}
\end{align}
and
\begin{equation}\label{Q06}
b_{k+1} > \sigma^{-\frac{4(n_{k+1}+1)}{2-\beta}}>\Bigg\{ \begin{array}{ll}
   \sigma^{-\frac{4(\alpha n_k+2)}{2-\beta}} &\quad  \text{if $k$  is even},\\
   \sigma^{-\frac{4(\beta n_k+2)}{2-\beta}} &\quad  \text{if $k$  s odd}. \end{array}
\end{equation}
Furthermore, it follows from \eqref{Q02} that $\varepsilon_k $ can be arbitrary small by taking $k$   sufficiently    large,  so there exists a positive integer $k_0$ such that for any  $k\geq k_0$ and $p\geq 0$, we get
\begin{equation*}
2\alpha^p\beta^p-n_k^{-1}+\dfrac{\log 2}{\log\varepsilon_k}>\alpha^{p+2}\beta^{p+2} . 
\end{equation*} 
Fix such a $k_0$. 
Then it   immediately holds that for any  $k\geq k_0$ and $p\geq 0$, 
\begin{equation}\label{Q121}
\varepsilon_k^{2\alpha^p\beta^p}<\varepsilon_k^{2\alpha^p\beta^p-n_k^{-1}}<\dfrac{1}{2}\varepsilon_k^{\alpha^{p+2}\beta^{p+2}}<\dfrac{1}{2}\varepsilon_k^{\alpha^{p+1}\beta^{p+1}}.
\end{equation}

In the following lemmas, we only consider the case when $k$ is an even number because it is enough to prove Theorem \ref{thm:main} and makes the statements a bit simpler,   but similar estimates   hold even when $k$ is an odd number. 
We first show the following.  
\begin{lem}\label{lem1}
For 
every even number $k\geq k_0$, $p\in \mathbb N \cup \{0\}$ and $j \in \{ 0 , 1\}$,  
\begin{align*}
\lambda^{n_{k+2p+j}}\sqrt{b_{k+2p+j}}\le \varepsilon_k^{\alpha^{p+j}\beta^p-n_k^{-1}}b_{k+2p+1+j}.
\end{align*}
\end{lem}
\begin{proof}
Fix an even number $k\geq k_0$. 
We will prove this lemma by induction with respect to $p$. 
For the case $p=0$,  it follows from \eqref{Q02}, \eqref{Q05} and \eqref{Q06} that
\begin{equation*}
\lambda^{n_k}\sqrt{ b_k}\le \lambda^{n_k}\sigma^{-2n_k}\le \Big(\lambda\sigma^{\frac{6\beta-4+8n_k^{-1}}{2-\beta}}\Big)^{n_k}\sigma^{-\frac{4(\alpha n_k+2)}{2-\beta}}<\varepsilon_k^{1-n_k^{-1}}\cdot b_{k+1},
\end{equation*}
and since $\frac{n_{k+1}}{n_k} \geq \frac{n_0 \alpha ^{k/2 +1}\beta ^{k/2} -1}{n_0 \alpha ^{k/2 }\beta ^{k/2} }\geq \alpha -\frac{1}{n_k}$ by the construction of $n_k$, 
\begin{equation*}
\begin{split}
\lambda^{n_{k+1}}\sqrt {b_{k+1}}
&\le \lambda^{n_{k+1}}\sigma^{-2n_{k+1}}\\
&\le \Big(\lambda\sigma^{\frac{6\beta-4+8n_{k+1}^{-1}}{2-\beta}}\Big)^{n_{k}\cdot\frac{n_{k+1}}{n_k}}\sigma^{-\frac{4(\beta n_{k+1}+2)}{2-\beta}}\le\varepsilon_k^{\alpha-n_k^{-1}}\cdot b_{k+2}.
\end{split}
\end{equation*}

Next we assume that the assertion of Lemma~\ref{lem1} is true  for a given $p\in \mathbb N\cup \{0\}$. 
Then we have
\begin{align*}
\lambda^{n_{k+2p+2}}\sqrt{b_{n+2p+2}}
&\le \lambda^{n_{k+2p+2}}\sigma^{-2n_{k+2p+2}}\\
&\le \Big(\lambda\sigma^{\frac{6\beta-4+8n_{k+2p+2}^{-1}}{2-\beta}}\Big)^{n_{k}\cdot\frac{n_{k+2p+2}}{n_k}}\sigma^{-\frac{4(\alpha n_{k+2p+2}+2)}{2-\beta}}\\
& <\varepsilon_k^{\alpha^{p+1}\beta^{p+1}-n_k^{-1}}\cdot b_{k+2p+3}
\end{align*}
and
\begin{align*}
\lambda^{n_{k+2p+3}}\sqrt{b_{n+2p+3}}
&\le \lambda^{n_{k+2p+3}}\sigma^{-2n_{k+2p+3}}\\
&\le \Big(\lambda\sigma^{\frac{6\beta-4+8n_{k+2p+3}^{-1}}{2-\beta}}\Big)^{n_{k}\cdot\frac{n_{k+2p+3}}{n_k}}\sigma^{-\frac{4(\beta n_{k+2p+3}+2)}{2-\beta}}\\
& <\varepsilon_k^{\alpha^{p+2}\beta^{p+1}-n_k^{-1}}\cdot b_{k+2p+4}.
\end{align*}
That is,   the assertion of Lemma~\ref{lem1}  with $p+1$ instead of $p$  is also true. 
This completes the induction and   the proof of Lemma~\ref{lem1}.
\end{proof}

Define a sequence  $(U_{k,m})_{m  \geq 0}$ of rectangles with $k\geq k_0$ by
\begin{equation}\label{eq:0915e2}
U_{k,m}=\left\{(x,y):\ |x|\leq \varepsilon_k^{(\alpha\beta)^{\lfloor \frac{m}{2} \rfloor }}b_{k+m},\ |y|\leq \varepsilon_k^{(\alpha\beta)^{\lfloor \frac{m}{2}\rfloor }}\sqrt{b_{k+m}}\right\} 
\end{equation}
for each integer $m\geq 0$.
Then, by Remark \ref{rmk:0911c}, $U_{k,m}$ is included in  $R_{k+m}$ for any large $m$ (under the translation of $(\tilde x_{k+m},0)$ to $(0,0)$), on which  $F_{k+m}$ in  \eqref{map0} is well-defined. 
 Then we have the following.
 \begin{lem}\label{lem2}
   For any even number $k\geq k_0$,  $m\in \mathbb N\cup \{0\}$ and $\boldsymbol{x}  \in U_{k,0}$, 
   \[
 F_{k+m-1} \circ F_{k+m-2} \circ \cdots \circ F_{k}(\boldsymbol{x}) 
\in U_{k, m}.
   \]
\end{lem}
\begin{proof}
Fix an even number $k\geq k_0$, an integer $m\geq 0$ and $\boldsymbol{x}  \in U_{k,0}$,
and 
set 
$
\boldsymbol{x}_{k,m} := F_{k+m-1} \circ F_{k+m-2} \circ \cdots \circ F_{k}(\boldsymbol{x}),
$
where $\boldsymbol{x}_{k,0}$ is interpreted as  $ \boldsymbol{x}$ so that $\boldsymbol{x}_{k,0}\in U_{k,0}$.
Denote the first and second coordinate of $\boldsymbol{x}_{k,m} $ by $x_{k,m}$ and $y_{k,m}$, respectively.  
We will show that $(x_{k,m} ,y_{k,m}) \in U_{k,m}$ by induction with respect to $m\in \mathbb N\cup \{0\}$.

We first show that
 $(x_{k,m},y_{k,m})\in U_{k,m}$ for $m=1$. 
It holds that 
\begin{align*}
|x_{k,1}| 
=|-\sigma^{2n_{k }}x_{k }^2 \mp \lambda^{n_{k}}y_{k }| 
\le \sigma^{4n_{k }} \varepsilon_k ^2b_{k }^2+\lambda^{n_{k }} \varepsilon_k \sqrt{b_{k }} 
\le  \varepsilon_k ^2b_{k +1}+ \varepsilon_k^{2 -n_k^{-1}}b_{k +1}.
\end{align*}
In the last inequality, the first term is due to the equality  
$\sigma^{4n_k}b_k^2= b_{k+1}$ implied by the definition of $b_k$,
and the second term comes from Lemma~\ref{lem1}. 
Hence, it follows from (\ref{Q121}) that
\begin{equation*}
|x_{k,1}|\le\dfrac{1}{2}\varepsilon_k^{\alpha\beta }b_{k+1}+\dfrac{1}{2}\varepsilon_k^{ \alpha\beta } b_{k+1}=\varepsilon_k^{ \alpha\beta }b_{k+1}\le  \varepsilon_k b_{k+1}
\end{equation*}
and
\begin{equation*}
|y_{k,1}|=|\sigma^{n_{k}}x_{k}|\le \sigma^{2n_{k}} \varepsilon_k b_{k}=\varepsilon_k \sqrt{b_{k+1}},
\end{equation*}
which concludes that $(x_{k,1},y_{k,1})\in U_{k+1}$.

Next we assume that   $(x_{k,m},y_{k,m})\in U_{k,m}$ for $m = 2p$ and $2p+1$ with a given integer $p\geq 0$. 
 In addition, we assume (as an inductive hypothesis) that  
\begin{equation*}
|x_{k,2p+1}|\le \varepsilon_k^{(\alpha\beta)^{p+1}}b_{k+2p+1},
\end{equation*}
which indeed holds in the case when $p=0$  as seen above. 
Then it holds that 
\begin{align*}
|x_{k,2p+2}|
&=|-\sigma^{2n_{k+2p+1}}x_{k,2p+1}^2\mp \lambda^{n_{k+2p+1}}y_{k,2p+1}|\nonumber\\
&\le \sigma^{4n_{k+2p+1}}\varepsilon_k^{2(\alpha\beta)^{p}}b_{k+2p+1}^2+\lambda^{n_{k+2p+1}}\varepsilon_k^{(\alpha\beta)^{p}}\sqrt{b_{k+2p+1}}\nonumber\\
&\le \varepsilon_k^{2(\alpha\beta)^{p}}b_{k+2p+2}+\varepsilon_k^{(\alpha\beta)^{p}}\cdot \varepsilon_k^{\alpha^{p+1}\beta^p-n_k^{-1}}b_{k+2p+2}\\
&\le\dfrac{1}{2}\varepsilon_k^{(\alpha\beta)^{p+1}}b_{k+2p+2}+\dfrac{1}{2}\varepsilon_k^{(\alpha\beta)^{p+1}} b_{k+2p+2}=\varepsilon_k^{(\alpha\beta)^{p+1}}b_{k+2p+2},\\
|y_{k,2p+2}|
&=|\sigma^{n_{k+2p+1}}x_{k,2p+1}|\\
&\le \sigma^{2n_{k+2p+1}}\varepsilon_k^{(\alpha\beta)^{p+1}}b_{k+2p+1}=\varepsilon_k^{(\alpha\beta)^{p+1}}\sqrt{b_{k+2p+2}},\\
|x_{k,2p+3}|
&=|-\sigma^{2n_{k+2p+2}}x_{k,2p+2}^2 \mp \lambda^{n_{k+2p+2}}y_{k,2p+2}|\nonumber\\
&\le \sigma^{4n_{k+2p+2}}\varepsilon_k^{2(\alpha\beta)^{p+1}}b_{k+2p+2}^2+\lambda^{n_{k+2p+2}}\varepsilon_k^{(\alpha\beta)^{p+1}}\sqrt{b_{k+2p+2}}\nonumber\\
&\le \varepsilon_k^{2(\alpha\beta)^{p+1}}b_{k+2p+3}+\varepsilon_k^{(\alpha\beta)^{p+1}}\cdot \varepsilon_k^{(\alpha\beta)^{p+1}-n_k^{-1}}b_{k+2p+3}\\
&\le\dfrac{1}{2}\varepsilon_k^{(\alpha\beta)^{p+3}}b_{k+2p+3}+\dfrac{1}{2}\varepsilon_k^{(\alpha\beta)^{p+3}} b_{k+2p+3}\\
&
\le \varepsilon_k^{(\alpha\beta)^{p+2}}b_{k+2p+3}
\le \varepsilon_k^{(\alpha\beta)^{p+1}}b_{k+2p+3},\\
|y_{k,2p+3}|&=|\sigma^{n_{k+2p+2}}x_{k,2p+2}|\\
&\le \sigma^{2n_{k+2p+2}}\varepsilon_k^{(\alpha\beta)^{p+1}}b_{k+2p+2}=\varepsilon_k^{(\alpha\beta)^{p+1}}\sqrt{b_{k+2p+3}}.
\end{align*}
This shows that $(x_{k,m},y_{k,m})\in U_{k,m}$ for  $m=2p+2$ and $2p+3$, which complete the proof of Lemma~\ref{lem2}.
\end{proof}

Since  $0<\lambda\sigma^{\frac{6\beta-4+8n_k^{-1}}{2-\beta}}<1$ for any $k\geq 0$ by \eqref{Q02}, there exists a positive integer $m'$ such that
\begin{equation*}
\log\lambda\sigma^\alpha>(\alpha\beta)^{\frac{m'}{2}}\log\Big(\lambda\sigma^{\frac{6\beta-4+8n_k^{-1}}{2-\beta}}\Big)
\end{equation*}
for any $k\geq 0$. 
Fix such an $m'$.
Fix also a real number $\xi\in(0,1)$.

\begin{lem}\label{lem4}
There exist positive integers  $k_1\geq k_0$ and  $m_0$ such that for any even number $k\geq k_1$, any integer $m\ge m_0$ and any $\boldsymbol{x} \in U_{k,0}$, 
\[
2|x_{k,m}|\sigma ^{2n_{k+m}}\le \xi\lambda^{n_{k+m}},
\]
where $x_{k,m}$ is  the first coordinate of $ F_{k+m-1} \circ F_{k+m-2} \circ \cdots \circ F_{k}(\boldsymbol{x} )$. 
\end{lem}

\begin{proof}
Since $\varepsilon _k$ goes to zero as $k\to \infty$, 
there exists an even number $k_1\geq k_0$
such that 
\[
\varepsilon _k \leq \varepsilon_{k_0}^{(\alpha\beta)^{\frac{m'+2}{2}}} 
\]
for any $k\geq k_1$.
Recall  that $k_0$ is   an even number. 
Note that
\[
n_k^{-1}(\alpha\beta)^{-\frac{m+1}{2}}(\log (2\lambda^{-1}\sigma)-\log\xi) \to 0 \quad \text{as $m\to \infty$},
\]
so by the choice of $m'$, 
there exists  an $m_0\in\mathbb{N}$ such that for every $m\ge m_0$,
\begin{equation*}
\log\lambda\sigma^\alpha\ge(\alpha\beta)^{\frac{m'}{2}}\log\Big(\lambda\sigma^{\frac{6\beta-4+8n_{k_0}^{-1}}{2-\beta}}\Big)+n_{k_0}^{-1}(\alpha\beta)^{-\frac{m+1}{2}}(\log (2\lambda^{-1}\sigma)-\log\xi).
\end{equation*}
Multiply the inequality 
  by $(\alpha\beta)^{\frac{m+1}{2}}$, then we get
\[
(\alpha\beta)^{\frac{m+1}{2}}\log\lambda\sigma^\alpha+n_{k_0}^{-1}\log\xi 
\ge (\alpha\beta)^{\frac{m'+m+1}{2}}\log\Big(\lambda\sigma^{\frac{6\beta-4+8n_{k_0}^{-1}}{2-\beta}}\Big)+n_{k_0}^{-1}\log (2\lambda^{-1}\sigma).
\]
Hence, it follows that 
\[
\xi^{n_{k_0}^{-1}}(\lambda\sigma^\alpha)^{(\alpha\beta)^{\lceil \frac{m}{2}\rceil}}\ge
\xi^{n_{k_0}^{-1}}(\lambda\sigma^\alpha)^{(\alpha\beta)^{\frac{m+1}{2}}}\ge (2\lambda^{-1}\sigma)^{n_{k_0}^{-1}}\Big(\lambda\sigma^{\frac{6\beta-4+8n_{k_0}^{-1}}{2-\beta}}\Big)^{(\alpha\beta)^{\frac{m'+m+1}{2}}}
\]
because $\frac{m+1}{2}\ge\lceil\frac{m}{2}\rceil$, 
where $\lceil x\rceil$ denotes the smallest integer which is larger than or equal to $x$.
Raise the above inequality  to   the   $n_{k_0}$-th power,  together with \eqref{Q02}, then we have
\begin{align*}
\lambda\sigma^{-1}\xi(\lambda\sigma^\alpha)^{n_{k_0}(\alpha\beta)^{\lceil \frac{m}{2}\rceil}}
&\ge
2\Big( \lambda\sigma^{\frac{6\beta-4+8n_{k_0}^{-1}}{2-\beta}} \Big)^{n_{k_0}(\alpha\beta)^{\frac{m'+m+1}{2}}}
= 2\big(\varepsilon_{k_0}^{(\alpha\beta)^{\frac{m'+2}{2}}}\big)^{(\alpha\beta)^{\frac{m-1}{2}}} \\
&\ge 2\big(\varepsilon_{k_0}^{(\alpha\beta)^{\frac{m'+2}{2}}}\big) ^{(\alpha\beta)^{\lfloor\frac{m}{2}\rfloor}}
\ge 2 \varepsilon_k   ^{(\alpha\beta)^{\lfloor\frac{m}{2}\rfloor}}
\end{align*}
for any $k\geq k_1$.

Fix an even number $k\geq k_1$ and an integer $m\geq m_0$. 
Then, due to Lemma~\ref{lem2}, 
 the definition of $b_{k+m}$, (\ref{Q05}) and the above inequality, we have
\begin{align}\label{eq:0914c}
\notag2|x_{k,m}|\sigma^{2n_{k+m}} &\le 2\varepsilon_ k^{(\alpha\beta)^{\lfloor\frac{m}{2}\rfloor}}b_{k+m}\sigma^{2n_{k+m}}\\
\notag &=2\varepsilon_k ^{(\alpha\beta)^{\lfloor\frac{m}{2}\rfloor}}\sqrt{b_{k+m+1}}\\
& \le  
\notag 2\varepsilon_k^{(\alpha\beta)^{\lfloor\frac{m}{2}\rfloor}}\sigma^{-2n_{k+m+1}} \le 
\notag 2\varepsilon_k^{(\alpha\beta)^{\lfloor\frac{m}{2}\rfloor}}\sigma^{-n_{k+m+1}}\\
&\le \lambda\sigma^{-1}\xi(\lambda\sigma^\alpha)^{n_k(\alpha\beta)^{\lceil \frac{m}{2}\rceil}}\sigma^{-n_{k+m+1}}.
\end{align}

On the other hand, when $m=2p$,  
\[
\lambda^{n_{k+m}}>\lambda^{n_k\alpha^p\beta^p+1} \quad \mbox{and}\quad \sigma^{n_{k+m+1}}>\sigma^{n_k\alpha^{p+1}\beta^p-1},
\] 
thus
\[
\lambda^{n_{k+m}}\sigma^{n_{k+m+1}}>\lambda\sigma^{-1}(\lambda\sigma^\alpha)^{n_k\alpha^p\beta^p}.
\]
Similarly, when $m=2p+1$,  
\[
\lambda^{n_{k+m}}>\lambda^{n_k\alpha^{p+1}\beta^p+1} \ \mbox{and}\ \sigma^{n_{k+m+1}}>\sigma^{n_k\alpha^{p+1}\beta^{p+1}-1}>\sigma^{n_k\alpha^{p+2}\beta^p-1},
\] 
thus
\[
\lambda^{n_{k+m}}\sigma^{n_{k+m+1}}>\lambda\sigma^{-1}(\lambda\sigma^\alpha)^{n_k\alpha^{p+1}\beta^p}>\lambda\sigma^{-1}(\lambda\sigma^\alpha)^{n_k\alpha^{p+1}\beta^{p+1}}.
\]
Therefore, we have
\begin{equation*}
\lambda^{n_{k+m}}=(\lambda^{n_{k+m}}\sigma^{n_{k+m+1}})\sigma^{-n_{k+m+1}}
\ge \lambda\sigma^{-1}(\lambda\sigma^\alpha)^{n_k(\alpha\beta)^{\lceil m/2\rceil}} \sigma^{-n_{k+m+1}}.
\end{equation*}
Combining this estimate with \eqref{eq:0914c}, we get
\[
2|x_{k,m}|\sigma ^{2n_{k+m}}\le \xi\lambda^{n_{k+m}},
\]
which completes the proof of Lemma~\ref{lem4}.
\end{proof}

\subsection{Lyapunov irregularity}\label{s:4.3}

Let $k_1$ and $m_0$ be integers given in the previous subsection, 
and we fix even numbers $k\geq k_1$ and $m\geq m_0$ throughout this  subsection.  

Fix $\boldsymbol{x}\in U_{k,0}$ and define $\boldsymbol{x}_{k, m+j} = (x_{k,m+j},y_{k,m+j})$ for each $j\geq 0$  by
\[
\boldsymbol{x}_{k,m+j}  := F_{k+m +j  -1} \circ  F_{k+m +j -2} \circ \cdots \circ  F_{k}(\boldsymbol{x}) .
\]
Recall Lemma \ref{lem4} for $\xi \in (0,1)$, and set 
\begin{equation*}
K :=\dfrac{1}{3\xi   }.
\end{equation*}
Fix also a vector $\boldsymbol{v}_0=(v_0,w_0) \in T_{\boldsymbol{x}_{k,m} }M$ with 
\begin{equation}\label{eq:0915c}
K^{-1} \leq \frac{|v_0|}{|w_0|}\leq K,
\end{equation}
 and inductively define $\boldsymbol{v}_{j}=(v_{j},w_{j}) $ for each $j\geq 0$ by
\[
\boldsymbol{v}_{j+1}  := DF_{k+m +j  }(\boldsymbol{x}_{k, m+j}) \boldsymbol{v}_{j}.
\]
For notational simplicity,  we below use
\[
 \kappa := k +m
 \]
  and  
\[
(n_p;n_{p+2q}):=n_p+n_{p+2}+n_{p+4}+\cdots +n_{p+2q}
\]
  for each $p, q\in\mathbb{N}$.
  For simplicity, we let $(n_p;n_{p-2})=0$ for $p\in \mathbb N$.
\begin{lem}\label{lem5}
There exist constants $C_j $ ($j =-2, -1, \ldots $)   such that 
\begin{align*}
\boldsymbol{v}_{2p}&=\left(                 
  \begin{array}{c}   
    v_{2p} \\  
    w_{2p} \\  
  \end{array}
\right)=\left(                 
  \begin{array}{c}   
    C_{2p-1}\lambda^{(n_{\kappa +1};n_{\kappa +2p-1})}\sigma^{(n_{\kappa };n_{\kappa +2p-2})}v_0 \\  
   \pm  C_{2p-2}\lambda^{(n_{\kappa };n_{\kappa +2p-2})}\sigma^{(n_{\kappa +1};n_{\kappa +2p-1})}w_0\\  
  \end{array}
\right)\\
\boldsymbol{v}_{2p+1}&=\left(                 
  \begin{array}{c}   
    v_{2p+1} \\  
    w_{2p+1} \\  
  \end{array}
\right)=\left(                 
  \begin{array}{c}   
   C_{2p}\lambda^{(n_{\kappa };n_{\kappa +2p})}\sigma^{(n_{\kappa +1};n_{\kappa +2p-1})}w_0 \\  
  \pm  C_{2p-1}\lambda^{(n_{\kappa +1};n_{\kappa +2p-1})}\sigma^{(n_{\kappa };n_{\kappa +2p})}v_0\\  
  \end{array}
\right)
\end{align*}
 for every $p\geq 0$, 
and that $\dfrac{1}{2}\leq |C_j| \leq \dfrac{3}{2}$ for every $j\geq  -2$.
\end{lem}
\begin{proof}
We prove Lemma \ref{lem5} by induction. 
We first show the claim for $p=0$. 
The formula for $\boldsymbol{v}_0$ obviously holds with $C_{-2}=C_{-1}=1$.
Due to \eqref{map0}, we have
\begin{equation*}
DF_\kappa (\boldsymbol{x}_{k,m} )= \left(                 
  \begin{array}{cc}   
    -2\sigma^{2n_\kappa }x_{k,m} &\mp \lambda^{n_\kappa }  \\  
  \pm  \sigma^{n_\kappa }      &  0             \\  
  \end{array}
\right),
\end{equation*}
and
\[
\left(                 
  \begin{array}{c}   
    v_1 \\  
    w_1 \\  
  \end{array}
\right) 
= 
DF_\kappa (\boldsymbol{x}_{k,m} )\left(                 
  \begin{array}{c}   
    v_0 \\  
    w_0 \\  
  \end{array}
\right)=\left(                 
  \begin{array}{c}   
    -2x_{k,m} \sigma^{2n_{\kappa }}v_0\mp \lambda^{n_{\kappa }}w_0 \\  
   \pm \sigma^{n_{\kappa }}v_0 \\  
  \end{array}
\right).
\]
By Lemma~\ref{lem4},  \eqref{eq:0915c} and the definition of $K$,
\[
|2x_{k,m} \sigma^{2n_{\kappa }}v_0|\le \xi\lambda^{n_{\kappa }}|v_0|\le\xi K \lambda^{n_{\kappa }}|w_0|
=\frac{1}{3} \lambda^{n_{\kappa }}|w_0|.
\]
In other words, 
\[
\left(                 
  \begin{array}{c}   
    v_1 \\  
    w_1 \\  
  \end{array}
\right)=\left(                 
  \begin{array}{c}   
    C_0\lambda^{n_{\kappa }}w_0 \\  
    \pm C_{-1}\sigma^{n_{\kappa }}v_0 \\  
  \end{array}
\right),
\]
with a constant  $C_0$ satisfying 
\begin{equation}\label{eq:0915d4}
 1-\frac{1}{3}   \le|C_0|\le 1+\frac{1}{3}.
 \end{equation}

Next we assume that the claim is true for a given $p\geq 0$, and will show the claim with $p+1$ instead of $p$. 
Note that
\begin{align*}
\boldsymbol{v}_{2p+2}
&=
\left(                 
  \begin{array}{c}   
    v_{2p+2} \\  
    w_{2p+2} \\  
  \end{array}
\right)
=DF_{\kappa + 2p +1}(\boldsymbol{x}_{k ,m + 2p+1})
\left(                 
  \begin{array}{c}   
    v_{2p+1} \\  
    w_{2p+1} \\  
  \end{array}
\right)\\
&=\left(                 
  \begin{array}{c}   
    -2x_{k ,m + 2p+1}\sigma^{2n_{\kappa + 2p +1}}v_{2p+1} \mp \lambda^{n_{\kappa + 2p +1}}w_{2p+1} \\  
    \pm \sigma^{n_{\kappa + 2p +1}}v_{2p+1} \\  
  \end{array}
\right),
\end{align*}
whose first coordinate is
\begin{multline*}
   - 2x_{k ,m + 2p+1}\sigma^{2n_{\kappa + 2p +1}} C_{2p}\lambda^{(n_{\kappa };n_{\kappa +2p})}\sigma^{(n_{\kappa +1};n_{\kappa +2p-1})}w_0 \\
 \mp
 C_{2p-1}\lambda^{(n_{\kappa +1};n_{\kappa +2p+1})}\sigma^{(n_{\kappa };n_{\kappa +2p})}v_0 ,
    \end{multline*}
    and second coordinate is
\begin{align*}
      \pm  C_{2p}  \lambda^{(n_{\kappa };n_{\kappa +2p})}\sigma^{(n_{\kappa +1};n_{\kappa +2p+1})}w_0,
    \end{align*}
 by the inductive hypothesis. 
On the other hand, it follows from Lemma~\ref{lem4}, \eqref{eq:0915c} and the monotonicity of $(n_l)_{l\in \mathbb N}$ that the absolute value of the first term of the first coordinate is bounded by
\begin{align*}
& \xi\lambda^{n_{\kappa + 2p +1}}|C_{2p}| \lambda^{(n_{\kappa };n_{\kappa +2p})}\sigma^{(n_{\kappa +1};n_{\kappa +2p-1})}|w_0|\\
&   \quad  \leq  \xi  K \lambda^{n_{\kappa + 2p +1}}
   |C_{2p}| \lambda^{(n_{\kappa -1};n_{\kappa +2p-1})}\sigma^{(n_{\kappa  +2};n_{\kappa +2p})}|v_0|
\\
&   \quad =  \frac{1}{3} \frac{\lambda ^{n_{\kappa -1}}}{\sigma ^{n_\kappa }}    |C_{2p}| \lambda^{(n_{\kappa +1};n_{\kappa +2p+1})}\sigma^{(n_{\kappa  };n_{\kappa +2p})}|v_0|\\
 &   \quad \leq     \frac{1}{3}   |C_{2p}| \lambda^{(n_{\kappa +1};n_{\kappa +2p+1})}\sigma^{(n_{\kappa  };n_{\kappa +2p})}|v_0|.
  \end{align*}
Hence, we can write $\boldsymbol{v}_{2p+2}$ as
\[
\left(                 
  \begin{array}{c}   
    v_{2p+2} \\  
    w_{2p+2} \\  
  \end{array}
\right)
=\left(                 
  \begin{array}{c}   
    C_{2p+1}\lambda^{(n_{\kappa +1};n_{\kappa +2p+1})}\sigma^{(n_{\kappa  };n_{\kappa +2p})}v_0 \\  
  \pm  C_{2p} \lambda^{(n_{\kappa };n_{\kappa +2p})}\sigma^{(n_{\kappa +1};n_{\kappa +2p+1})} w_0 \\  
  \end{array}
\right),
\]
with a constant $C_{2p +1}$ satisfying that
\begin{equation}\label{eq:0915d2}
|C_{2p-1}|- \frac{1}{3}  |C_{2p}| \leq |C_{2p+1}| \leq |C_{2p-1}|+  \frac{1}{3}   |C_{2p}|.
\end{equation}

Similarly, 
\begin{align*}
\boldsymbol{v}_{2p+3}
&=
\left(                 
  \begin{array}{c}   
    v_{2p+3} \\  
    w_{2p+3} \\  
  \end{array}
\right)
=DF_{\kappa + 2p +2}(\boldsymbol{x}_{k ,m + 2p+2})
\left(                 
  \begin{array}{c}   
    v_{2p+2} \\  
    w_{2p+2} \\  
  \end{array}
\right)\\
&=\left(                 
  \begin{array}{c}   
   - 2x_{k ,m + 2p+2}\sigma^{2n_{\kappa + 2p +2}}v_{2p+2}\mp \lambda^{n_{\kappa + 2p +2}}w_{2p+2} \\  
   \pm \sigma^{n_{\kappa + 2p +2}}v_{2p+2} \\  
  \end{array}
\right),
\end{align*}
whose first coordinate is
\begin{multline*}
   - 2x_{k ,m + 2p+2}\sigma^{2n_{\kappa + 2p +2}} C_{2p+1}\lambda^{(n_{\kappa +1};n_{\kappa +2p+1})}\sigma^{(n_{\kappa  };n_{\kappa +2p})}v_0 \\
 \mp C_{2p}\lambda^{(n_{\kappa  };n_{\kappa +2p+2})}\sigma^{(n_{\kappa +1};n_{\kappa +2p+1})}w_0 ,
    \end{multline*}
    and second coordinate is
\begin{align*}
  \pm      C_{2p+1}  \lambda^{(n_{\kappa +1};n_{\kappa +2p+1})}\sigma^{(n_{\kappa };n_{\kappa +2p+2})}v_0
    \end{align*}
    by the previous formula. 
On the other hand, it follows from Lemma~\ref{lem4}, \eqref{eq:0915c} and the monotonicity of $(n_l)_{l\in \mathbb N}$ that the absolute value of the first term of the first coordinate is bounded by
\begin{align*}
& \xi\lambda^{n_{\kappa + 2p +2}}|C_{2p+1}| \lambda^{(n_{\kappa +1};n_{\kappa +2p+1})}\sigma^{(n_{\kappa };n_{\kappa +2p})}|v_0|\\
&     \quad  \leq  \xi  K 
 \lambda^{n_{\kappa + 2p +2}}   |C_{2p+1}| \lambda^{(n_{\kappa };n_{\kappa +2p})}\sigma^{(n_{\kappa  +1};n_{\kappa +2p +1})}|w_0|
\\
&    \quad =    \frac{1}{3}   |C_{2p+1}| \lambda^{(n_{\kappa };n_{\kappa +2p+2})}\sigma^{(n_{\kappa  +1};n_{\kappa +2p+1})}|w_0|.
  \end{align*}
Hence, we can write $\boldsymbol{v}_{2p+3}$ as
\[
\left(                 
  \begin{array}{c}   
    v_{2p+3} \\  
    w_{2p+3} \\  
  \end{array}
\right)
=\left(                 
  \begin{array}{c}   
    C_{2p+2}\lambda^{(n_{\kappa  };n_{\kappa +2p+2})}\sigma^{(n_{\kappa +1};n_{\kappa +2p+1})}w_0  \\  
  \pm  C_{2p+1}  \lambda^{(n_{\kappa +1};n_{\kappa +2p+1})}\sigma^{(n_{\kappa };n_{\kappa +2p+2})}v_0 \\  
  \end{array}
\right),
\]
with a constant $C_{2p +2}$ satisfying that
\begin{equation}\label{eq:0915d3}
|C_{2p}|-  \frac{1}{3}   |C_{2p+1}| \leq |C_{2p+2}| \leq |C_{2p}|+\frac{1}{3}   |C_{2p+1}|.
\end{equation}

Finally, combining \eqref{eq:0915d4}, \eqref{eq:0915d2} and \eqref{eq:0915d3}, we get
\[
 \frac{1}{2} = 1- \left(\frac{1}{3} + \frac{1}{3^2} +\cdots \right) \leq \vert C_j \vert \leq 1+ \left(\frac{1}{3} + \frac{1}{3^2} +\cdots \right) = \frac{3}{2}
\]
for any $j\geq 0$.
This completes the proof of Lemma \ref{lem6}.
\end{proof}

Given two sequences $(a_p)_{p\geq 0}$ and $(b_p)_{p\geq 0}$ of positive  numbers, if there exist constants $c_0,\ c_1>0$, independently of $p$, such that
\[c_0<\dfrac{a_p}{b_p}<c_1,\]  
then, we say that $a_p$ and $b_p$ are equivalent, denoted by $a_p\sim b_p$.

\begin{lem}\label{lem6}
For every $p\geq 0$, we have  
\begin{align*}
|v_{2p}|\sim \lambda^{(n_{\kappa +1};n_{\kappa +2p-1})}\sigma^{(n_{\kappa };n_{\kappa +2p-2})}&<\lambda^{(n_{\kappa };n_{\kappa +2p-2})}\sigma^{(n_{\kappa +1};n_{\kappa +2p-1})}\sim |w_{2p}|,\\
|v_{2p+1}|\sim \lambda^{(n_{\kappa };n_{\kappa +2p})}\sigma^{(n_{\kappa +1};n_{\kappa +2p-1})}&<\lambda^{(n_{\kappa +1};n_{\kappa +2p-1})}\sigma^{(n_{\kappa };n_{\kappa +2p})}\sim |w_{2p+1}|. 
\end{align*}
\end{lem}
\begin{proof}
The equivalence relations follow from Lemma~\ref{lem5} directly. Since $0<\lambda<1$, $\sigma>1$,
\[
\lambda^{(n_{\kappa +1};n_{\kappa +2p-1})}<\lambda^{(n_{\kappa };n_{\kappa +2p-2})},\quad \sigma^{(n_{\kappa };n_{\kappa +2p-2})}<\sigma^{(n_{\kappa +1};n_{\kappa +2p-1})},
\]
which gives the former formula immediately. In order to prove the later formula, it suffice to notice that 
\[
\lambda^{(n_{\kappa };n_{\kappa +2p})}<\lambda^{(n_{\kappa +2};n_{\kappa +2p})}<\lambda^{(n_{\kappa +1};n_{\kappa +2p-1})},
\]
\[
\sigma^{(n_{\kappa +1};n_{\kappa +2p-1})}<\sigma^{(n_{\kappa +2};n_{\kappa +2p})}<\sigma^{(n_{\kappa };n_{\kappa +2p})}.
\]
This completes the proof of Lemma \ref{lem6}.
\end{proof}

An immediate consequence of Lemma~\ref{lem6} is that
\[
\|\boldsymbol{v}_{2p}\|\sim \lambda^{(n_{\kappa };n_{\kappa +2p-2})}\sigma^{(n_{\kappa +1};n_{\kappa +2p-1})},
\]
\[\
\|\boldsymbol{v}_{2p+1}\|\sim \lambda^{(n_{\kappa +1};n_{\kappa +2p-1})}\sigma^{(n_{\kappa };n_{\kappa +2p})}.
\]
Since both $k$ and $m$ are even numbers, for every integer $p\geq 0$, we have
\[
|n_{\kappa +2p}-n_0\alpha^{\frac{\kappa}{2}+p}\beta^{\frac{\kappa }{2}+p}|\le 1,\quad |n_{\kappa +2p+1}-n_0\alpha^{\frac{\kappa }{2}+p+1}\beta^{\frac{\kappa }{2}+p}|\le 1.
\]  
According to Lemma~\ref{lem6}, 
\begin{align*}
&
\lim\limits_{p\to \infty}
\dfrac{\log\|D(F_{n_{\kappa +2p}}\circ\cdots\circ F_{n_{\kappa +1}}\circ F_{n_{\kappa }})(\boldsymbol{x}_{k,m})\boldsymbol{v}_0\|}{(n_{\kappa }+2)+(n_{\kappa +1}+2)+\cdots+(n_{\kappa +2p}+2)}\\
& \quad =
\lim\limits_{p\to\infty}
\dfrac{\log\|\boldsymbol{v}_{2p+1}\|}{n_{\kappa }+n_{\kappa +1}+\cdots+n_{\kappa +2p}+O(p)}\\
& \quad  =
\lim\limits_{p\to \infty} 
\dfrac{\log\lambda^{(n_{\kappa +1};n_{\kappa +2p-1})}\sigma^{(n_{\kappa };n_{\kappa +2p})} +O(p)}{n_{\kappa }+n_{\kappa +1}+\cdots+n_{\kappa +2p}+O(p)}\\
& \quad  =
\lim\limits_{p\to \infty}
 \dfrac{(n_{\kappa +1}+n_{\kappa + 3}+\cdots+n_{\kappa +2p-1})\log\lambda+(n_{\kappa }+n_{\kappa +2}+\cdots+n_{\kappa +2p})\log\sigma +O(p)}{n_{\kappa }+n_{\kappa +1}+\cdots+n_{\kappa +2p}+O(p)}\\
&\quad  =
\lim\limits_{p\to \infty}
 \dfrac{n_0\alpha^{\frac{\kappa }{2}}\beta^{\frac{\kappa }{2}}[\alpha(1+\alpha\beta+\cdots+ (\alpha\beta)^{p-1})\log\lambda+(1+\alpha\beta+\cdots+(\alpha\beta)^p)\log\sigma]+O(p)}{n_0\alpha^{\frac{\kappa }{2}}\beta^{\frac{\kappa }{2}}[1+\alpha+\alpha\beta+\alpha^2\beta+\cdots+(\alpha\beta)^p] +O(p)}\\
& \quad  = \dfrac{\log\lambda+\beta\log\sigma}{1+\beta},
\end{align*}
and
\begin{align*}
&
\lim\limits_{p\to \infty}
\dfrac{\log\|D(F_{n_{\kappa +2p-1}}\circ\cdots\circ F_{n_{\kappa +1}}\circ F_{n_{\kappa }})(\boldsymbol{x}_{k,m})\boldsymbol{v}_0\|}{(n_{\kappa }+2)+(n_{\kappa +1}+2)+\cdots+(n_{\kappa +2p-1}+2)}\\
&\quad  =
\lim\limits_{p\to\infty}
\dfrac{\log\|\boldsymbol{v}_{2p}\|}{n_{\kappa }+n_{\kappa +1}+\cdots+n_{\kappa +2p-1}+O(p)}\\
&\quad  =
\lim\limits_{p\to \infty}
 \dfrac{\log\lambda^{(n_{\kappa };n_{\kappa +2p-2})}\sigma^{(n_{\kappa +1};n_{\kappa +2p-1})} +O(p)}{n_{\kappa }+n_{\kappa +1}+\cdots+n_{\kappa +2p-1}+O(p)}\\
&\quad  =
\lim\limits_{p\to \infty}
 \dfrac{(n_{\kappa }+n_{\kappa +2}+\cdots+n_{\kappa +2p-2})\log\lambda+(n_{\kappa +1}+n_{\kappa + 2p +1}+\cdots+n_{\kappa +2p-1})\log\sigma +O(p)}{n_{\kappa }+n_{\kappa +1}+\cdots+n_{\kappa +2p-1}+O(p)}\\
&\quad  =
\lim\limits_{p\to \infty} 
\dfrac{n_k\alpha^{\frac{m}{2}}\beta^{\frac{m}{2}}[\alpha(1+\alpha\beta+\cdots+ (\alpha\beta)^{p-1})\log\lambda+(1+\alpha\beta+\cdots +(\alpha\beta)^{p-1})\log\sigma]+O(p)}{n_k\alpha^{\frac{m}{2}}\beta^{\frac{m}{2}}[1+\alpha+\alpha\beta+\alpha^2\beta+\cdots+(\alpha\beta)^{p-1}+\alpha^p\beta^{p-1}]+O(p)}\\
&\quad  = \dfrac{\log\lambda+\alpha\log\sigma}{1+\alpha}.
\end{align*}
Since
 \[\dfrac{\log\lambda+\beta\log\sigma}{1+\beta}\not=\dfrac{\log\lambda+\alpha\log\sigma}{1+\alpha},\] 
together with the remark in the end of Subsection \ref{s:4.1},
this completes the proof of the assertion for the Lyapunov irregularity in Theorem \ref{thm:main}, where  $U_f$ and $V_f$ are the interiors of  $F_{\kappa -1} \circ F_{\kappa -2} \circ \cdots \circ F_{k}(U_{k,0})$ (under the coordinate translation) and $\{ (v_0, w_0) \mid K^{-1} \leq \frac{\vert v_0\vert }{\vert w_0\vert } \leq K\}$, respectively.

\subsection{Birkhoff (ir)regularity}\label{s:4.4}
To show Birkhoff (ir)regularity, as well as the uncountability   of   $f$  up to conjugacy in Theorem \ref{thm:main}, we need a more detailed description of 
Colli-Vargas' theorem as follows.
Let $g$ be the surface diffeomorphism of Theorem \ref{prop:0812} and 
\[
\mathbb B_+^u:=g([-1,1]^2) \cap ([0,1]\times [-1,1]), \quad \mathbb B_-^u:=g([-1,1]^2) \cap ([-1,0]\times [-1,1]).
\]
Fo each $l\in \mathbb{N}$ and $\underline w=(w_{1}, w_2, \ldots ,w_{l}) \in \{ +, -\}^{l}$, we let
\begin{align*}
\mathbb{B}^{u}_{\underline w}:=\bigcap _{j=1}^l g^{-j+1} (\mathbb B_{w_j}^u), \quad 
\mathbb{G}^{u}_{\underline w}:=\mathbb{B}^{u}_{\underline w} \setminus \left(\mathbb{B}^{u}_{\underline w+}\cup \mathbb{B}^{u}_{\underline w-}\right),
\end{align*}
where $\underline w\pm  =(w_1,\ldots ,w_l, \pm )\in \{+,-\}^{l+1}$.
\begin{thm}[\cite{CV2001}]\label{thm:0911b}
Let $g $ be the surface diffeomorphism with a robust homoclinic tangency given in Theorem \ref{prop:0812}.
Take
\begin{itemize}
\item a   $\mathcal C^r$-neighborhood $\mathcal O$ of $g$ with $2\leq r<\infty$, 
\item an increasing sequence $(n_k^0)_{k\in \mathbb N}$ of integers satisfying $n_{k}^0 =O((1+\eta )^k) $ with some  $\eta >0$,  
\item a  sequence $(\underline z _k^0)_{k\in \mathbb N}$ of codes with $\underline z_k^0 \in \{+,-\}^{n_k^0}$. 
\end{itemize} 
Then, one can find 
\begin{itemize}
\item a   diffeomorphism $f$ in $\mathcal O$ which coincides with $g$ on $\mathbb B_+^u \cup \mathbb B_-^u$, 
\item a sequence of rectangles $(R_k)_{k\in \mathbb N}$, 
\item  increasing sequences $(\hat n_k)_{k\in \mathbb N}$, $(\hat m_k)_{k\in \mathbb N}$ of integers satisfying that $\tilde n_k := \hat n_{k}+\hat m_{k+1} =O(k) $ and depends only on $\mathcal O$,
\item  sequences $(\hat{\underline z} _k)_{k\in \mathbb N}$, $(\hat{ \underline w} _k)_{k\in \mathbb N}$ of codes with $\hat{\underline  z}_k \in \{+,-\}^{\hat n_k}$, $\check{\underline  w}_k \in \{+,-\}^{\hat m_k}$
\end{itemize}
 such that 
 for each $k\in \mathbb N$, 
(a), (b) in Theorem  \ref{prop:0812} hold   and
\begin{itemize}
\item[$\mathrm{(c)}$]  $R_k \subset \mathbb G_{\underline z_k}^u$ for  $\underline z_k=\hat{\underline  z}_k   \underline  z_k ^0 [\hat{\underline  w}_{k+1} ]^{-1} $, where  $[\underline w]^{-1} = (w_{l},\ldots ,w_{2}, w_1)$ for each $\underline w =(w_1, w_2,\ldots ,w_l) \in \{ +,-\} ^l$, $l\in \mathbb N$.
\end{itemize} 
\end{thm}
Fix a neighborhood $\mathcal O$ of $g$ 
and a sequence $(n_k^0)_{k\in \mathbb N}$ as given in  \eqref{eq:0915f}.
To indicate the dependence of $\boldsymbol{z} = (\underline z _k^0)_{k\in \mathbb N}$ on $f$ and  $(R_k)_{k\in \mathbb N}$ in Theorem \ref{thm:0911b}, we write them as $f_{\boldsymbol{z}}$ and  $(R_{k,\boldsymbol{z}})_{k\in \mathbb N}$.

We first apply Theorem \ref{thm:0911b} to the sequence  $\boldsymbol{z} = (\underline z _k^0)_{k\in \mathbb N}$ given by 
\[
\underline z_k^0=(+,+,\ldots ,+,   z'_k), \quad z'_k\in \{ +, -\}
\]
for each $k\geq 1$.
Then, it is straightforward to see from the item (c) of Theorem \ref{thm:0911b} that for any $k\in \mathbb N$, continuous function  $\varphi : M\to \mathbb R$ and $ \epsilon >0$, there exist integers $k_2$ and $L_0$ such that
\[
\sup _{\boldsymbol{x}\in R_k} \left\vert \varphi (f^{n}_{\boldsymbol{z} }(\boldsymbol{x} ) )-  \varphi( \boldsymbol{p} _+) \right\vert <\epsilon
\]
whenever    
\[
N(k,k') + L_0 \leq n  \leq N(k,k'+1) - L_0
\]
with some $k'\geq k_2$, where $\boldsymbol{p} _+$ is the continuation for $f_{\boldsymbol{z} }$ of the saddle fixed point of $g$ corresponding to the point set  $\mathbb B_{(+, +, \ldots )}^u$ and 
\[
N(p,q):=\sum _{k=p}^q (n_k +2)
\]
 for each $p ,q \in \mathbb N$ with $p\leq q$.  
Hence, 
  it holds that
\[
\lim _{n\to\infty} \frac{1}{n}\sum_{j=0}^{n-1} \varphi ( f^j_{\boldsymbol{z} }(\boldsymbol{x} )) = \varphi (\boldsymbol{p} _+)
\]
for any  $k\in \mathbb N$,  $\boldsymbol{x} \in R_k$ and continuous function  $\varphi : M\to \mathbb R$.
Since the open set $U_{f_{\boldsymbol{z}}} $ consisting of Lyapunov irregular points constructed in the previous subsection is of the form $f_{\boldsymbol{z}} ^n (U_{0,k'})$ with some positive integers $n$ and $k'$, it follows from the remark following \eqref{eq:0915e2} and the item (a) of Theorem \ref{thm:0911b} that 
$U_{f_{\boldsymbol{z}}}  \subset R_k$ with some $k$. 
 This implies that any point in $U_{f_{\boldsymbol{z}}} $ is Birkhoff regular.

Notice  that the choice of $(z'_1, z'_2, \ldots ) $ in $\boldsymbol{z}$ is uncountable. 
On the other hand, 
if $\boldsymbol{z} =(\underline z_k^0)_{k\in \mathbb N}$ and $\boldsymbol{w} = (\underline w_k^0)_{k\in \mathbb N}$ are of the above form (in particular, $\underline w_k^0=(+,+,\ldots ,+,   w'_k) \in \{ +,-\} ^{n_k^0}$ with $w'_k\in \{ +, -\}$) and $z_k' \neq w_k'$ for some 
 $k$, 
  then 
$f_{\boldsymbol{z} }$ and $f_{\boldsymbol{w} }$ are not topologically conjugate, or $f_{\boldsymbol{z} }$ and $f_{\boldsymbol{w} }$ are   topologically conjugate by a homeomorphism $h$ on $M$ and $h(R_{k, \boldsymbol{z}}) \cap R_{k, \boldsymbol{w}} = \emptyset$ for every $k$, because   of the item (c) of Theorem \ref{thm:0911b} and the fact that both $f_{\boldsymbol{z} }$ and $f_{\boldsymbol{w} }$ coincide with $g$ on $\mathbb B_+^u \cup \mathbb B_-^u$.
Therefore, since  there can exist at most countably many, mutually disjoint open sets  (of positive Lebesgue measure) on $M$ due to the compactness of $M$, 
we complete the proof of the claim for the uncountable set $\mathcal R$ in Theorem \ref{thm:main}.

We next apply Theorem \ref{thm:0911b} to the sequence  $\boldsymbol{z} = (\underline z _k^0)_{k\in \mathbb N}$ given by 
\[
\underline z_k^0=
\begin{cases}
(+,+,\ldots ,+,z_k') \quad &\text{if $(2p -1)^2\leq k< (2p )^2$ with some $p$}\\
(-,-,\ldots ,-,z_k') \quad & \text{if $(2p )^2\leq k< (2p +1)^2$ with some $p$}
\end{cases}
\]
with $z'_k\in \{ +, -\}$
for each $k\geq 1$.
Then, it follows from the item (c) of Theorem \ref{thm:0911b} that for any $k\in \mathbb N$, continuous function  $\varphi : M\to \mathbb R$ and $ \epsilon >0$, there exist integers $k_2$ and $L_0$ such that
\[
\sup _{\boldsymbol{x}\in R_k} \left\vert \varphi (f^{n}_{\boldsymbol{z} }(\boldsymbol{x} ) )-  \varphi( \boldsymbol{p} _+) \right\vert <\epsilon
\]
whenever    
\[
N(k,k') + L_0 \leq n  \leq N(k,k'+1) - L_0, \quad \max\{ k_2 , (2p -1)^2\} \leq k'< (2p)^2
\]
with some $p$, 
and
\[
\sup _{\boldsymbol{x}\in R_k} \left\vert \varphi (f^{n}_{\boldsymbol{z} }(\boldsymbol{x} ) )-  \varphi( \boldsymbol{p} _-) \right\vert <\epsilon
\]
whenever    
\[
N(k,k') + L_0 \leq n  \leq N(k,k'+1) - L_0, \quad \max\{ k_2 , (2p )^2\} \leq k'< (2p+1)^2
\]
with some $p$, where $\boldsymbol{p} _-$ is the continuation for $f_{\boldsymbol{z} }$ of the saddle fixed point of $g$ corresponding to the point set  $\mathbb B_{(-, -, \ldots )}^u$.
Hence, if we let 
\[
\mathbf N(\ell ):= N(k,(\ell +1)^2) - N(k,\ell ^2) = \sum _{k=\ell ^2+1}^{(\ell +1)^2} (n_k +2), 
\]
then for
any $k\in \mathbb N$, $\boldsymbol{x}\in R_k$ and continuous function  $\varphi : M\to \mathbb R$, we have
\[
  \frac{1}{\mathbf N(2p-1)} \sum_{j= N(k,(2p -1)^2)}^{  N(k,(2p )^2)-1} \varphi (f^{j}_{\boldsymbol{z} }(\boldsymbol{x} ) ) =    \varphi ( \boldsymbol{p} _+)  +o(1)
\]
and
\[
  \frac{1}{\mathbf N(2p)} \sum_{j= N(k,(2p )^2)}^{N(k,(2p+1 )^2)-1} \varphi (f^{j}_{\boldsymbol{z} }(\boldsymbol{x} ) )  =    \varphi ( \boldsymbol{p} _-)  +o(1).
\]
Since $\mathbf N(1 ) + \mathbf N(2 ) +\cdots + \mathbf N(\ell -1) 
=o(\mathbf N(\ell ) ) $, this implies that, with $\ell := \lceil \sqrt k\rceil $ which we assume to be an odd number for simplicity, 
\begin{align*}
&\frac{1}{N(k,(2p +1)^2)}\sum_{j=0}^{N(k,(2p +1)^2)-1} \varphi (f^{j}_{\boldsymbol{z} }(\boldsymbol{x} ) )   \\
& \quad =\frac{1}{N(k,(2p +1)^2) -N(k, \ell ^2)}\sum_{j=N(k, \ell ^2)}^{N(k,(2p +1)^2)-1} \varphi (f^{j}_{\boldsymbol{z} } (\boldsymbol{x} ) )  +o(1) \\
& \quad =  \frac{ \mathbf N(\ell ) + \mathbf N(\ell +2) + \cdots + \mathbf N(2p -1)}{\mathbf N(\ell ) +\mathbf N(\ell +1) + \cdots + \mathbf N(2p )}  \varphi ( \boldsymbol{p} _+)
 + \frac{ \mathbf N(\ell +1) + \mathbf N(\ell +3) + \cdots + \mathbf N(2p)}{\mathbf N(\ell ) +\mathbf N(\ell +1) + \cdots + \mathbf N(2p)}  \varphi ( \boldsymbol{p} _-)  +o(1) \\
 & \quad \to  \varphi ( \boldsymbol{p} _+) \quad (p\to \infty ).
\end{align*}
Similarly we have
\[
\lim _{p\to \infty}\frac{1}{N(k,(2p )^2)}\sum_{j=0}^{N(k,(2p )^2)-1} \varphi (f^{j}_{\boldsymbol{z} }(\boldsymbol{x} ) ) = \varphi ( \boldsymbol{p} _-). 
\]
That is, any point in $R_k$ is Birkhoff irregular.
Therefore, repeating the argument for $\mathcal R$, 
 we obtain the claim for the uncountable set $\mathcal I$ in Theorem \ref{thm:main}.
This completes the proof of Theorem \ref{thm:main}.

\begin{rem}
The proof of Birkhoff (ir)regularity in this subsection  essentially appeared in Colli-Vargas \cite{CV2001}.
The   difference is that our $(n_k^0)_{k\in \mathbb N}$ increases exponentially fast because of the requirement \eqref{eq:0915f}, while  their  $(n_k^0)_{k\in \mathbb N}$ is  of order $O(k^2)$.
  \end{rem}

\appendix

\section{Lebesgue measurability of irregular sets}

Although it might be a folklore theorem, we have never seen the proof that Birkhoff and  Lyapunov irregular sets are Lebesgue measurable.
  In this appendix we show that   Birkhoff and  Lyapunov irregular sets are Lebesgue measurable as a corollary of the following proposition.
\begin{prop}
Let $T$ be a Polish space and $(\theta _n)_{n\in \mathbb N}$   a sequence of functions from $M\times T$ to $\mathbb R$. 
Then the irregular set $I$ of $(\theta _n)_{n\in \mathbb N}$ over $T$, given by
\[
I=\left\{ x\in M \mid \text{there exists $t\in T$ for which $\displaystyle \lim_{n\to\infty} \theta _n(x,t)$ does not exist}\right\},
\]
is a Lebesgue measurable set of $M$. 
\end{prop}
For simplicity, we assume that $M$ is an open subset of $\mathbb R^d$ and identify $TM$ with $M\times \mathbb R^d$. 
The Birkhoff irregular set of a continuous map $f: M\to M$ is the irregular set  of $(\theta _n)_{n\in \mathbb N}$,
\[
\theta _n(x,\varphi ):= \frac{1}{n}\sum _{j=0}^{n-1} \varphi \circ f^j(x) \quad ((x,\varphi )\in M\times \mathcal C^0(M)),
\]
 over $T=\mathcal C^0(M)$, i.e.~the space of all continuous functions on $M$, 
and  the Lyapunov irregular set of a differentiable map $f: M\to M$ is the irregular set  of $(\theta _n)_{n\in \mathbb N}$,  
\[
\theta _n(x, v):= \frac{1}{n}\log \Vert Df^n(x)v\Vert \quad ((x,v)\in M\times (\mathbb R^d\setminus \{0\})),
\]
over $T=\mathbb R^d \setminus \{0\}$.
\begin{proof}
We first note that $I$ is the projection of 
\[
\widehat I:= \left\{(x,t)\in M\times T   \mid \text{$\displaystyle \lim_{n\to\infty} \theta _n(x,t)$ does not exist} \right\} 
\]
along the Polish space $T$, 
that is,
\[
I=\left\{x\in M \mid \text{there exists $t\in T$ such that $(x,t)\in \widehat I$}\right\}.
\]

We will show that $\widehat I$ is a Borel set.  
For each $n\in \mathbb N$, $\alpha , \beta \in \mathbb R$, we define open sets $A_n (\alpha )$ and $B_n (\beta )$ of $M\times T$ by
\begin{align*}
&A_{n}(\alpha ) =\left\{ (x,t)\in M\times T \mid   \theta _n(x,t)  >\alpha   \right\},\\
&B_{n}(\beta ) =\left\{ (x,t)\in M\times T \mid   \theta _n(x,t)  <\beta   \right\}.
\end{align*}
Notice that
\[
\left\{(x,t)\in M\times T \mid \limsup_{n\to\infty} \theta _n(x,t) \geq \alpha \right\} 
=\bigcap _{n_0\in \mathbb N}\bigcup _{n\geq n_0}A_{n}(\alpha )
\]
and 
\[
\left\{(x,t)\in M\times T \mid \liminf_{n\to\infty} \theta _n(x,t) \leq \beta \right\} 
=\bigcap _{n_0\in \mathbb N}\bigcup _{n\geq n_0}B_{n}(\beta ).
\]
Hence, we get that
\[
\widehat I=\bigcup _{\substack{(\alpha ,\beta ) \in \mathbb Q^2\\ \beta <\alpha }}  \left( \left(\bigcap _{n_0\in \mathbb N}\bigcup _{n\geq n_0}  A_n(\alpha ) \right) \cap \left(  \bigcap _{n_0\in \mathbb N}\bigcup _{n\geq n_0}   B_n(\beta)\right) \right),
\]
which implies that $\widehat I$ is a Borel set, as claimed.

Due to the well-known facts that   every projection of a Borel set along a Polish space is   an analytic set (i.e.~the image of a continuous map from a Polish space $X$ to $T$), and that any analytic set is Lebesgue measurable, 
  the irregular set $I$ is a Lebesgue measurable set.
\end{proof}
\begin{rem}
From the above proof, the Birkhoff irregular set for each   $\varphi \in \mathcal C^0(M)$ and the Lyapunov irregular set for each  $v\in \mathbb R^d\setminus \{0\}$ are Borel measurable, while  it is unclear whether the Birkhoff and  Lyapunov irregular sets of $f$ are Borel measurable because we might need to consider a  non-denumerable union of Borel measurable sets to find these irregular sets.
\end{rem}

\section*{Acknowledgments}
We are sincerely grateful to Masayuki Asaoka for many fruitful discussions and telling us that the positivity of Lebesgue measure of Lyapunov irregular sets is transmitted by conjugacies. 
We are also grateful
to the anonymous referee for  valuable  comments. 
This work was partially supported by JSPS KAKENHI
Grant Numbers 21K03332, 19K14575, 19K21834  and 18K03376
and by National Natural Science Foundation of China  Grant Numbers 11701199.

\end{document}